\documentclass[12pt,reqno]{amsart}
\usepackage{hyperref}
\usepackage{amssymb,amsmath,amsfonts,latexsym}
\usepackage{bm}
\usepackage{mathtools}
\usepackage{enumerate}
\usepackage{array,graphics,color}
\allowdisplaybreaks
\usepackage{mathrsfs}
\usepackage{mathtools}
\usepackage{blkarray, bigstrut}

\numberwithin{equation}{section}
\newtheorem{dfn}{Definition}[section]
\newtheorem{thm}[dfn]{Theorem}
\newtheorem{lma}[dfn]{Lemma}

\DeclarePairedDelimiterX{\norm}[1]{\lVert}{\rVert}{#1}
\DeclarePairedDelimiterX{\bnorm}[1]{\big\lVert}{\big\rVert}{#1}
\DeclarePairedDelimiterX{\Bnorm}[1]{\Big\lVert}{\Big\rVert}{#1}
\newcommand\at[2]{\left.#1\right|_{#2}}

\newcommand{\D}{\mathbb{D}}

\newcommand{\hil}{\mathscr{H}}

\newcommand{\boh}{\mathcal{B}_1(\hil)}
\newcommand{\dds}{\dfrac{d}{ds}}

\newcommand{\cir}{\mathbb{T}}

\newcommand{\bnh}{\mathcal{B}_n(\mathscr{H})}
\newcommand{\bnf}{\mathcal{B}_n(\mathscr{F})}
		
\newcommand{\HDTzero}{\mathbf{H}^2_{\mathcal{D}_{T_0}}(\mathbb{D})}
		
\newcommand{\HDTzeros}{\mathbf{H}^2_{\mathcal{D}_{T_0^*}}(\mathbb{D})}
\newcommand{\HDT}{\mathbf{H}^2_{\mathcal{D}_{T}}(\mathbb{D})}	
\newcommand{\HDTs}{\mathbf{H}^2_{\mathcal{D}_{T^*}}(\mathbb{D})}
\newcommand{\PH}{P_{\mathscr{H}}}
\newcommand{\PF}{P_{\mathscr{F}}}
\newcommand{\DT}{\mathcal{D}_T}
\newcommand{\DTs}{\mathcal{D}_{T^*}}

\newcommand{\DTzero}{\mathcal{D}_{T_0}}
\newcommand{\DTzeros}{\mathcal{D}_{T_0^*}}

\begin{document}
	
	\title[Higher-order spectral shift for pairs of contractions]{Higher-order spectral shift for pairs of contractions via multiplicative path}
	
	\author[Chattopadhyay] {Arup Chattopadhyay}
	\address{Department of Mathematics, Indian Institute of Technology Guwahati, Guwahati, 781039, India}
	\email{arupchatt@iitg.ac.in, 2003arupchattopadhyay@gmail.com}
	
	\author[Pradhan]{Chandan Pradhan}
	\address{Department of Mathematics, Indian Institute of Technology Guwahati, Guwahati, 781039, India}
	\email{chandan.math@iitg.ac.in, chandan.pradhan2108@gmail.com}

	
	\begin{abstract}
		In \cite{Mor}, Marcantognini and Mor\'{a}n
		obtained Koplienko-Neidhardt trace formula for pairs of contractions and pairs of maximal dissipative operators via multiplicative path. In this article, we prove the existence of higher-order spectral shift functions for pairs of contractions and pairs of maximal dissipative operators via multiplicative path by adapting the argument employed in \cite{Mor}.
	\end{abstract}
	\maketitle

	\section{Introduction}
	The spectral shift function (SSF)
has become a fundamental object in perturbation theory. The notion of first order spectral shift function originated from the work of the Russian physicist I.M. Lifshits' on theoretical physics \cite{Lif}, followed by  M.G. Krein in \cite{Kr53, Kr83}, in which it was shown that for a pair
	of self-adjoint (not necessarily bounded) operators $H$ and $H_0$ satisfying $H-H_0\in \mathcal{B}_1(\mathscr{H})$ (set of trace class operators on a separable Hilbert space $\mathscr{H}$), there exists a unique real-valued $L^1(\mathbb{R})$- function $\xi$  such that 
	\begin{equation}\label{inteq1}
		\text{Tr}~\{\phi(H)-\phi(H_0)\} = \int_{\mathbb{R}} \phi'(\lambda)~\xi(\lambda)~d\lambda,
	\end{equation}
	for a large class of functions $\phi$. The function $\xi$ is known as Krein's spectral shift function and the relation \eqref{inteq1} is called Krein's trace formula. A similar result was obtained by Krein in \cite{Kr62} for pair of unitary operators $\big\{U,U_0\big\}$ such that $U-U_0\in \mathcal{B}_1(\mathscr{H})$. For each such pair there exists a real valued $L^1([0,2\pi])$- function $\xi$, unique modulo an additive constant, such that
	\begin{equation}\label{intequ2}
		\text{Tr}~\big\{\phi(U)-\phi(U_0)\big\} = \int_{0}^{2\pi} \frac{d}{dt}\big\{\phi(\textup{e}^{it})\big\}~\xi(t)~dt,
	\end{equation}
	where $\phi'$ has absolutely convergent Fourier series. The original proof of Krein uses analytic function theory and for various alternative proofs of the formula \eqref{inteq1} and \eqref{intequ2} we refer to \cite{BS1,BS2,MoSi94,MoSi96,Voi}. Moreover, for a description of a wider class of functions for which formulae \eqref{inteq1} and \eqref{intequ2} hold we refer to \cite{AP16,Pe16}. For a pair of contractions $T_1$, $T_0$ with $T_1-T_0$ is trace-class, Neidhardt \cite{N,  NTh} initiated the study of trace formula, to be followed by others in \cite{AdNei,ChSiCont,MNP19}.
	In this connection, it is worthwhile to mention that a series of papers by Rybkin \cite{Ry1,Ry2,Ry3,Ry4}, where an analogous extension of \eqref{inteq1} and \eqref{intequ2} in case of contractions was also achieved.
	
	The modified second-order spectral shift function in the case of non-trace class perturbations was introduced by Koplienko in \cite{Ko}. Let $H$ and $H_0$ be two self-adjoint operators in a separable Hilbert space $\mathscr{H}$ such that $H-H_0=V\in \mathcal{B}_2(\mathscr{H})$ (set of Hilbert-Schmidt operators on $\mathscr{H}$). In this case, the difference $\phi(H)-\phi(H_0)$ is no longer of trace-class, and one has to consider instead
	$
	\phi(H)-\phi(H_0)-\at{\dfrac{d}{ds}\Big(\phi(H_0+sV)\Big)}{s=0},
	$
	where $\at{\dfrac{d}{ds}\Big(\phi(H_0+sV)\Big)}{s=0}$ denotes the G$\hat{a}$teaux derivative of $\phi$ at $H_0$ in the direction $V$ (see \cite{RB}) and find a trace formula for the above expression under certain assumptions on $\phi$. Under this hypothesis, Koplienko's formula asserts that there exists a unique function $\eta \in L^1(\mathbb{R})$ such that
	\begin{equation}\label{intequ3}
		\operatorname{Tr}\Big\{\phi(H)-\phi(H_0)-\at{\dfrac{d}{ds}\Big(\phi(H_0+sV)\Big)}{s=0}\Big\}=\int_{\mathbb{R}} \phi''(\lambda)~\eta(\lambda)~d\lambda
	\end{equation}
	for rational functions $\phi$ with poles off $\mathbb{R}$.
	The function $\eta$ is known as Koplienko spectral shift function corresponding to the pair $(H_0,H)$. For more on Koplienko trace formula we refer to \cite{ChSi,DS,GePu, SK10} and the references cited therein. In this connection, it is worth mentioning for a higher order version of \eqref{intequ3} we refer to \cite{PoSkSu13In}.

	A similar problem for unitary operators was considered by Neidhardt \cite{NH}. Let $U$ and $U_0$ be two unitary operators on a separable Hilbert space $\hil$ such that $U-U_0~\in~\mathcal{B}_2(\hil)$. Then  $U=e^{iA}U_0$, where $A$ is a self-adjoint operator in $\mathcal{B}_2(\hil)$. Set $U_s=e^{isA}U_0,~s\in \mathbb{R}$. Then it was shown in \cite{NH} that there exists an $L^1([0,2\pi]))$-function $\eta$ (unique upto an additive constant) such that 
	\begin{equation}\label{intequ4}
		\operatorname{Tr}\Big\{\phi(U)-\phi(U_0)-\at{\dfrac{d}{ds}\big\{\phi(U_s)\big\}}{s=0}\Big\}=\int_{0}^{2\pi} \dfrac{d^2}{dt^2} \big\{\phi (e^{it})\big\} \eta(t) dt,
	\end{equation}
	where $\phi''$ has absolutely convergent Fourier series.   It is important to note that the path considered by Neidhard \cite{NH} is a unitary path (or sometimes it is known as multiplicative path), that is  $U_s=e^{isA}U_0$ is an unitary operator for each $s\in\mathbb{R}$. Later, in \cite[Sect.10]{GePu}, Gesztesy, Pushnitski and Simon have discussed an alternative to Neidhardt's approach. In other words they have considered the linear path $U_0+ s(U-U_0);~0\leq s\leq 1$  and proved that there exists a real distribution $\eta$ on the unit circle $\mathbb{T}$ so that the formula \eqref{intequ4} holds 
	for every complex polynomial and also posed an open question \cite[Open Question 11.2]{GePu} whether the distribution $\eta$ becomes an $L^1(\mathbb{T})$-function or not? In 2012, Potapov and Sukochev gave an affirmative answer of the question in \cite{PoSu} and proved that
	given a pair of contractions $(T_1,T_0)$ in $\mathscr{H}$ such that $T_1-T_0\in \mathcal{B}_2(\mathscr{H})$, then there exists an $L^1(\mathbb{T})$-function $\eta$  (unique up to an analytic term) such that
	\begin{align}\label{Potasukoeq}
		\operatorname{Tr}\Bigg\{p(T)-p(T_0)-\at{\dds}{s=0} \big\{p(T_s)\big\} \Bigg\}=\int_{\cir}
		p''(z)\eta(z)dz,
	\end{align}
where $T_s=T_0+s(T_1-T_0),~s\in [0,1]$ and $p(\cdot)$ is any complex polynomial. Moreover, for a description of a wider class of functions for which formulae \eqref{intequ3} and \eqref{intequ4} hold we refer to \cite{Pe05}. Furthermore, for a higher order version of \eqref{intequ4} and \eqref{Potasukoeq} and an alternative proof via finite dimensional approximations of the formulae \eqref{intequ4} and \eqref{Potasukoeq} we refer to \cite{PoSkSu, PoSkSu16, Sk,ST} and \cite{CDP,CDP1} respectively. 	

In this direction of studies, Marcantognini and Mor\'{a}n 
obtained the Koplienko-Neidhardt trace formula for pairs of contraction operators and pairs of maximal dissipative operators via multiplicative path in \cite{Mor}. The aim of the present article is to prove a higher order version of  the Koplienko-Neidhardt trace formula for pairs of contraction operators and pairs of maximal dissipative operators via multiplicative path. We follow the same strategy as mentioned in \cite{Mor} with an appropriate modification. In other words, first we consider a pair $(T,V)$, where $V$ is a unitary operator and $T$ is a contraction on $\mathscr{H}$ and then we prove the higher order version of the Koplienko-Neidhardt trace formula via multiplicative path corresponding to the pair $(T,V)$ under some additional hypotheses  (see Theorem~\ref{thconuni}) by applying the existing Theorem~\ref{th1} below and using dilation theory.
Later, we prove the higher order version of the Koplienko-Neidhardt trace formula via multiplicative path for pairs of contrations $(T_1,T_0)$ (see Theorem~\ref{thconcon}) using our obtain Theorem~\ref{thconuni} in section 3. Finally, as an application of our main Theorem~\ref{thconcon} for pairs of contractions we obtain the higher order analogue of the Koplienko-Neidhardt trace formula via multiplicative path for pairs of maximal dissipative operators. 

The rest of the paper is organized as follows: Section 2 deals with some preliminaries that are essential in the later sections. In Section 3, we prove the higher order analogue of the Koplienko-Neidhard trace formula for a pair $(T,V)$, where $V$ is a unitary operator and $T$ is a contraction on $\mathscr{H}$. Section 4 is devoted to obtain the higher order version of the Koplienko-Neidhard trace formula for pairs of contractions. Consequently, in Section 5, we prove the trace formula for pairs of maximal dissipative operators.

	\section{Preliminaries}
	{\textbf{Notations:}} Here, $\mathscr{H}$ will denote the separable infinite dimensional Hilbert space we work in; $\mathcal{B}(\mathscr{H})$, $\mathcal{B}_1(\mathscr{H})$, $\mathcal{B}_2(\mathscr{H})$, $\mathcal{B}_n(\mathscr{H})$ the set of bounded, trace class, Hilbert-Schmidt class, Schatten-n class operators
	in $\mathscr{H}$ respectively with $\|\cdot\|$, $\|\cdot\|_1$, $\|\cdot\|_2$, $\|\cdot\|_n$ as the associated norms. Given $T\in \mathcal{B}(\mathscr{H})$, we denote its kernel by $\textup{Ker}(T)$, its range by $\textup{Ran}(T)$ and
	its spectrum by $\sigma(T)$, and let $\textup{Dom}(A)$, $\text{Tr}(A)$ be the domain
	of the operator $A$ and the trace of a trace class operator $A$ respectively. Also, $\mathbb{N}$, $\mathbb{Z}$, $\mathbb{R}$, and $\mathbb{C}$ denote the collection of natural, integer, real and complex numbers respectively. Furthermore, $\mathbb{D}$ stands for the open unit disk in the complex plane $\mathbb{C}$ and $\mathbb{T}$ for the unit circle in $\mathbb{C}$, hence $\mathbb{D}:= \big\{|z|<1,~z\in \mathbb{C}\big\}$ and $\mathbb{T}:= \big\{|z|=1,~z\in \mathbb{C}\big\}$. Further, given a closed subspace $\mathcal{M}$ of $\mathscr{H}$, $P_{\mathcal{M}}$ denotes the orthogonal projection of $\mathscr{H}$ onto $\mathcal{M}$. 
	
	Recall that $\mathscr{H}$-valued Hardy space over the unit disc $\mathbb{D}$ in $\mathbb{C}$ is denoted by $	\mathbf{H}^2_{\mathscr{H}}(\mathbb{D})$ and defined by 
	\begin{equation}
		\mathbf{H}^2_{\mathscr{H}}(\mathbb{D}):=\Big\{f(z)=\sum_{k=0}^{\infty}a_kz^k:~\|f\|_{\mathbf{H}^2_{\mathscr{H}}(\mathbb{D})}^2:=\sum_{k=0}^{\infty}\|a_k\|_{\mathscr{H}}^2<\infty,~z\in \mathbb{D},~a_k\in \mathscr{H}\Big\}.
	\end{equation}
Recall that the shift operator on the Hardy space $	\mathbf{H}^2_{\mathscr{H}}(\mathbb{D})$ is denoted by $S_{\mathscr{H}}$ and  is defined by 
$(S_{\mathscr{H}}f)(z):=zf(z),~f\in \mathbf{H}^2_{\mathscr{H}}(\mathbb{D}),~z\in \mathbb{D}$. It is easy to check that $S_{\mathscr{H}}$ is an isometry on $\mathbf{H}^2_{\mathscr{H}}(\mathbb{D})$ and $S_{\mathscr{H}}S_{\mathscr{H}}^*=I-P_{\mathscr{H}}$, where $P_{\mathscr{H}}$ is the orthogonal projection of $\mathbf{H}^2_{\mathscr{H}}(\mathbb{D})$ onto $\mathscr{H}$ (that is, by identifying $\mathscr{H}$ as $\mathscr{H}$-valued constant functions). For more on vector valued Hardy space we refer to \cite{Nikol,Parting}.

Let $T\in \mathcal{B}(\mathscr{H})$ be a contraction, that is $\|T\|\leq 1$. Then the defect operator of $T$ is denoted by $D_T$ and defined by $D_T:=(1-T^*T)^{1/2}$. Moreover, $\mathcal{D}_T:=\overline{\textup{Ran}(D_T)}$ is known as the corresponding defect space of $T$. Recall that the minimal unitary dilation of a contraction $T$ is a unitary operator $U_T:\mathscr{F}=\mathbf{H}^2_{\mathcal{D}_{T^*}}(\mathbb{D})\oplus\mathscr{H}\oplus\mathbf{H}^2_{\mathcal{D}_{T}}(\mathbb{D})\to\HDTs\oplus\mathscr{H}\oplus\HDT$
such that $T^n=P_{\mathscr{H}}U_T^n|_{\mathscr{H}}$ and $T^{*^n}=P_{\mathscr{H}}U_T^{-n}|_{\mathscr{H}}$
for $n\in \mathbb{N}$, and $\mathscr{F}$ is the smallest Hilbert space containing the subspaces $U_T^n\mathscr{H}$ for all $n\in \mathbb{Z}$. Furthermore, the block matrix representation of $U_T$ is as follows:
\begin{align}\label{dilationmat}
	U_T=\begin{bmatrix}
		S_{\DTs}^*&0&0\\
		D_{T^*}P_{\DTs}&T&0\\
		-T^*P_{\DTs}&D_{T}&S_{\mathcal{D}_T}
	\end{bmatrix}:\begin{bmatrix}
	\HDTs\\
	\mathscr{H}\\
	\HDT
\end{bmatrix}\longrightarrow \begin{bmatrix}
\HDTs\\
\mathscr{H}\\
\HDT
\end{bmatrix},
\end{align}
 where $S_{\mathcal{D}_T}$ and $S_{\mathcal{D}_{T^*}}$ are the shift operator on $\HDT$ and $\HDTs$ respectively and $P_{\DTs}$ is the orthogonal projection from $\mathscr{F}$ onto $\DTs\oplus {\textbf{0}}\oplus {\textbf{0}}\equiv \DTs$. For more on dilation theory we refer to \cite{NzFo}.

Next we introduce the set of functions $\mathcal{F}_n(\mathbb{T})$ for $n\in \mathbb{N}\cup \{0\}$:
\begin{equation}
\mathcal{F}_n(\mathbb{T}):=\Big\{f(z)=\sum_{k=-\infty}^{\infty}\hat{f}(k)z^k\in C^n(\mathbb{T}):\sum_{k=-\infty}^{\infty}|k|^n|\hat{f}(k)|<\infty \Big\},
\end{equation}
where $\big\{\hat{f}(k):~k\in \mathbb{Z}\big\}$ is the Fourier coefficients of $f$ and $C^n(\mathbb{T})$ is the collection of all $n$-times continuously differentiable functions on $\mathbb{T}$ and $f^{(n)}$ denotes the $n$-th order derivative of the function $f\in C^n(\mathbb{T})$. In particular for $n=0$, that is $\mathcal{F}_0(\mathbb{T})$ is known as the Wiener class on $\mathbb{T}$. Going further, we need the following well-known theorem due to A. Skripka (see \cite[Theorem 4.4]{Sk})
which gives the existence of higher order spectral shift for pairs of unitaries via multiplicative path to achieve our main results in later sections.

\begin{thm}(See \cite[Theorem 4.4]{Sk})\label{th1}
		Let $n\in\mathbb{N},~n\geq 2$. Let $U_0$ be a unitary operator, $A=A^*\in\bnh$ and denote $U_t=e^{itA}U_0,~t\in [0,1]$. Then for any $\phi\in\mathcal{F}_n(\mathbb{T})$,
		\begin{equation*}
			\left\{\phi(U_1)-\phi(U_0)-\sum_{k=1}^{n-1}\dfrac{1}{k!}\dfrac{d^k}{ds^k}\Big|_{s=0}\phi(U_s) \right\}\in \boh,
		\end{equation*}
		and there exists a constant $c_n$ and an $L^1([0, 2\pi])$-function $\eta_n=\eta_{n,U_0,A}$ (unique up to an additive constant) satisfying $\|\eta_n\|_1\leq c_n\|A\|_n^n$ such that 
		\begin{align*}
			\textup{Tr}\left\{\phi(U_1)-\phi(U_0)-\sum_{k=1}^{n-1}\dfrac{1}{k!}\dfrac{d^k}{ds^k}\Big|_{s=0}\phi(U_s) \right\}=\int_{0}^{2\pi}\dfrac{d^n}{dt^n}\big\{\phi(e^{it})\big\}\eta_n(t)dt.
		\end{align*}
	\end{thm}
	Let $\phi\in\mathcal{F}_n(\mathbb{T})$ be such that $\phi(e^{it})=\sum\limits_{k=-\infty}^{\infty}\hat{\phi}(k)e^{ikt}$. Next we introduce the functions, namely  $\phi_{+}(e^{it})=\sum\limits_{k=0}^{\infty}\hat{\phi}(k)e^{ikt}$ and $\phi_{-}(e^{it})=\sum\limits_{k=1}^{\infty}\hat{\phi}(-k)e^{ikt}$. Then $\phi(e^{it})=\phi_{+}(e^{it})+\phi_{-}(e^{-it})$ and $\phi_{\pm}\in \mathcal{F}_n(\mathbb{T})$. Now for a given contraction $T$ on $\mathscr{H}$, we set
		\begin{align}\label{INTeq1}
			\phi_{+}(T)=\sum\limits_{k=0}^{\infty}\hat{\phi}(k)T^k,~   \phi_{-}(T)=\sum\limits_{k=1}^{\infty}\hat{\phi}(-k){T^*}^k, \text{ and } \phi(T)=\phi_{+}(T)+\phi_{-}(T).
	\end{align}
	
	\section{Higher order Trace formula for pair of contractive operators with one of them is unitary}\label{conuni}
	In this section, we prove the higher order version of the Koplienko-Neidhardt trace formula via multiplicative path for a pair $(T,V)$, where $T$ is a contraction and $V$ is a unitary operator on $\mathscr{H}$ such that $T-V\in \mathcal{B}_n(\mathscr{H})$. 
	To proceed further, we need the following auxiliary lemma towards  obtaining our main result in this section. Note that the Lemma~\ref{lm1} below is available in \cite{ST} (see Theorem 5.3.4) in the case when $U$ is an unitary operator and the expression of the $k$-th order G$\hat{a}$teaux derivative of $f(U_t)$ is given in terms of multiple operator integral, where $f$ belongs to the Besov space. But for reader convenience we are providing a detailed proof herewith. On the other hand it is important to note that in our setting we consider {\emph{$U$ simply an element of $\mathcal{B}(\mathscr{H})$ instead of  unitary operator}} and we need the expression of the $k$-th order G$\hat{a}$teaux derivative of $p(U_t)$, where $p$ is any monomial. For that, it is essential to use the definition of the G$\hat{a}$teaux derivative (with convergence in the operator norm) in comparison to the idea of multiple operator integral used in \cite{ST} to get the expression of the $k$-th order G$\hat{a}$teaux derivative of $p(U_t)$ but at the same time to obtain the precise coefficients in the expression we are using the same idea as given in the proof of Theorem 5.3.4 in \cite{ST}.

\begin{lma}\label{lm1}
	Let $p(z)=z^n,~z\in \mathbb{T}$ and $n\in \mathbb{N}$, let $A\in \mathcal{B}(\mathscr{H})$ be a self-adjoint operator and let $U\in \mathcal{B}(\mathscr{H})$. Set $U_t=e^{itA}U,~t\in \mathbb{R}$. Then for all $1\leq k\leq n-1$, we have
	
	\begin{align}\label{eq2}
		\dfrac{d^k}{dt^k}\Big|_{t=s}\big\{U_t^n\big\}=\sum_{r=1}^{k}~\sum_{\substack{l_1+l_2+\cdots+l_{r}=k\\l_1,l_2,\ldots,l_{r}\geq 1}}~\dfrac{k!}{l_1!\cdots l_r!}\Bigg[\sum_{\substack{\alpha_0+\alpha_1+\cdots+\alpha_r=n-r\\\alpha_0,\alpha_1,\ldots,\alpha_r\geq 0}}U_s^{\alpha_0}W_s^{l_1}U_s^{\alpha_1}\cdots W_s^{l_{r}}U_s^{\alpha_r}\Bigg],
	\end{align} 
where $W_s^l=\Big((iA)^le^{isA}U\Big)$, $l\in\mathbb{N}$.
	\end{lma}
	\begin{proof}
		We prove the lemma by applying mathematical induction on $k$. For $k=1$, using the definition of G$\hat{a}$teaux derivative we get
		\begin{align*}
			\dfrac{d}{dt}\Big|_{t=s}\big\{U_t^n\big\}=&\lim_{h\to 0} \dfrac{U_{s+h}^n-U_s^n}{h}=\lim_{h\to 0}~\sum_{j=0}^{n-1}U_{s+h}^{n-j-1}\left(\dfrac{U_{s+h}-U_s}{h}\right)U_s^j\\
			=&\sum_{j=0}^{n-1}U_{s}^{n-j-1}\left(iAe^{isA}U\right)U_s^j=\sum_{\substack{\alpha_0+\alpha_1=n-1\\ \alpha_0,\alpha_1\geq 0}}U_{s}^{\alpha_0}W_s^1U_s^{\alpha_1}.
		\end{align*}

	Similarly for $k=2$, again by using the definition of G$\hat{a}$teaux derivative we have
		\begin{align*}
			\dfrac{d^2}{dt^2}\Big|_{t=s}\big\{U_t^n\big\}=&\sum_{\substack{\alpha_0+\alpha_1=n-1\\ \alpha_0\geq 1~\&~\alpha_1\geq 0}}\sum_{\substack{\beta_0+\beta_1=\alpha_0-1\\ \beta_0,\beta_1\geq 0}}U_s^{\beta_0}W_s^1U_s^{\beta_1}W_s^1U_s^{\alpha_1} \\
			& + \sum_{\substack{\alpha_0+\alpha_1=n-1\\ \alpha_1\geq 1~\&~\alpha_0\geq 0}}\sum_{\substack{\beta_0+\beta_1=\alpha_1-1\\ \beta_0,\beta_1\geq 0}}U_s^{\alpha_0}W_s^1U_s^{\beta_0}W_s^1U_s^{\beta_1}
			+\sum_{\substack{\alpha_0+\alpha_1=n-1\\ \alpha_0,\alpha_1\geq 0}}U_s^{\alpha_0}W_s^2U_s^{\alpha_1}\\	
			=&2!~\sum_{\substack{\alpha_0+\alpha_1+\alpha_2=n-2\\ \alpha_0,\alpha_1,\alpha_2\geq 0}}U_s^{\alpha_0}W_s^1U_s^{\alpha_1}W_s^1U_s^{\alpha_2} +\sum_{\substack{\alpha_0+\alpha_1=n-1\\ \alpha_0,\alpha_1\geq 0}}U_s^{\alpha_0}W_s^2U_s^{\alpha_1}\\
			=& \sum_{r=1}^{2}~\sum_{\substack{l_1+l_2+\cdots+l_{r}=2\\l_1,l_2,\ldots,l_{r}\geq 1}}~\dfrac{2!}{l_1!\cdots l_r!}\Bigg[\sum_{\substack{\alpha_0+\alpha_1+\cdots+\alpha_r=n-r\\\alpha_0,\alpha_1,\ldots,\alpha_r\geq 0}}U_s^{\alpha_0}W_s^{l_1}U_s^{\alpha_1}\cdots W_s^{l_{r}}U_s^{\alpha_r}\Bigg].
		\end{align*}
  Therefore the result is true for $k=1,2$. Now we assume that the result holds for $k=q< n-1$, that is the equation \eqref{eq2} is true for $k=q$. Next we show that the equation \eqref{eq2} also holds for $k=q+1$. Now by applying Leibnitz rule for G$\hat{a}$teaux derivative and using induction hypothesis we get
  \begin{align*}
  	&	\dfrac{d^{q+1}}{dt^{q+1}}\Big|_{t=s}\big\{U_t^n\big\}\\
  	=&\sum_{r=1}^{q}~\sum_{\substack{l_1+\cdots+l_{r}=q\\l_1,l_2,\ldots,l_{r}\geq 1}}~\dfrac{q!}{l_1!\cdots l_r!}	\dfrac{d}{dt}\Big|_{t=s}\Bigg[\sum_{\substack{\alpha_0+\cdots+\alpha_r=n-r\\\alpha_0,\alpha_1,\ldots,\alpha_r\geq 0}}U_t^{\alpha_0}W_t^{l_1}U_t^{\alpha_1}\cdots W_t^{l_{r}}U_t^{\alpha_r}\Bigg]\\
  	=&\sum_{r=1}^{q}~\sum_{\substack{l_1+\cdots+l_{r}=q\\l_1,\ldots,l_{r}\geq 1}}~\dfrac{q!}{l_1!\cdots l_r!}\\
  	&\hspace{.5in}\times\sum_{k=1}^{r}\Bigg[\sum_{\substack{\alpha_0+\cdots+\alpha_r=n-r\\\alpha_0,\ldots,\alpha_r\geq 0}}U_s^{\alpha_0}W_s^{l_1}U_s^{\alpha_1}\cdots W_s^{l_{k-1}}U_s^{\alpha_{k-1}}W_s^{l_k+1}U_s^{\alpha_k}W_s^{l_{k+1}}\cdots W_s^{l_{r}}U_s^{\alpha_r}\Bigg]\\
  	&+\sum_{r=1}^{q}~\sum_{\substack{l_1+\cdots+l_{r}=q\\l_1,\ldots,l_{r}\geq 1}}~\dfrac{q!}{l_1!\cdots l_r!}\\
  	&\times\sum_{k=1}^{r+1}\Bigg[\sum_{\substack{\alpha_0+\cdots+\alpha_{r+1}=n-(r+1)\\\alpha_0,\ldots,\alpha_{r+1}\geq 0}}U_s^{\alpha_0}W_s^{l_1}U_s^{\alpha_1}\cdots W_s^{l_{k-1}}U_s^{\alpha_{k-1}}W_s^1U_s^{\alpha_k} W_s^{l_{k}}U_s^{\alpha_{k+1}}\cdots W_s^{l_{r}}U_s^{\alpha_{r+1}}\Bigg]\\
  	=& K_1+K_2 \text{ (say)}.
  \end{align*} 
  Now if we substitute $j_a=l_a, 1\leq a\neq k\leq r$ and $j_k=l_k+1$ in $K_1$, we obtain
  \begin{align*}
  	K_1=&\sum_{r=1}^{q}~\sum_{k=1}^{r}~\sum_{\substack{j_1+\cdots+j_{r}=q+1\\j_a\geq 1, a\neq k,j_k\geq 2}}~\dfrac{q!}{j_1!\cdots j_{k-1}!(j_k-1)!\cdots j_r!}\\
  	&\hspace{.5in}\times\Bigg[\sum_{\substack{\alpha_0+\cdots+\alpha_r=n-r\\\alpha_0,\ldots,\alpha_r\geq 0}}U_s^{\alpha_0}W_s^{j_1}U_s^{\alpha_1}\cdots W_s^{j_{k-1}}U_s^{\alpha_{k-1}}W_s^{j_k}U_s^{\alpha_k}W_s^{j_{k+1}}\cdots W_s^{j_{r}}U_s^{\alpha_r}\Bigg].
  \end{align*}
  On the other hand, by relabeling the summands of $K_2$ via $r\mapsto r-1$ and performing the substitution $j_a = l_a , 1\leq a\leq k-1,j_k=1$, and  $j_a = l_{a-1}, k+1\leq a\leq r$, we obtain
  \begin{align*}
  	K_2=&\sum_{r=2}^{q+1}~\sum_{\substack{l_1+\cdots+l_{r-1}=q\\l_1,\ldots,l_{r-1}\geq 1}}~\dfrac{q!}{l_1!\cdots l_{r-1}!}\\
  	&\times\sum_{k=1}^{r}\Bigg[\sum_{\substack{\alpha_0+\cdots+\alpha_{r}=n-r\\\alpha_0,\ldots,\alpha_{r}\geq 0}}U_s^{\alpha_0}W_s^{l_1}U_s^{\alpha_1}\cdots W_s^{l_{k-1}}U_s^{\alpha_{k-1}}W_s^1U_s^{\alpha_k} W_s^{l_{k}}U_s^{\alpha_{k+1}}\cdots W_s^{l_{r-1}}U_s^{\alpha_{r}}\Bigg]\\
  	=&\sum_{r=2}^{q+1}~\sum_{k=1}^{r}~\sum_{\substack{j_1+\cdots+j_{r}=q+1\\j_a\geq 1, j_k=1}}~\dfrac{q!}{j_1!\cdots j_{k-1}!\cdot j_{k+1}!\ldots j_{r}!}\Bigg[\sum_{\substack{\alpha_0+\cdots+\alpha_{r}=n-r\\\alpha_0,\ldots,\alpha_{r}\geq 0}}U_s^{\alpha_0}W_s^{j_1}\cdots W_s^{j_{r}}U_s^{\alpha_{r}}\Bigg].
  \end{align*}
  Thus,
  \begin{align*}
  	K_1+K_2=&\sum_{\substack{\alpha_0+\alpha_1=n-1\\\alpha_0,\alpha_1\geq 0}}U_s^{\alpha_0}W_s^{q+1}U_s^{\alpha_1}\\
  	&+\sum_{r=2}^{q}~\sum_{k=1}^{r}~\Bigg[\sum_{\substack{j_1+\cdots+j_{r}=q+1\\j_a\geq 1, a\neq k,j_k\geq 2}}+\sum_{\substack{j_1+\cdots+j_{r}=q+1\\j_a\geq 1, j_k=1}}\Bigg]\dfrac{q!}{j_1!\cdots j_{k-1}!(j_k-1)!\cdots j_r!}\\
  	&\hspace{1.2in}\times \Bigg[\sum_{\substack{\alpha_0+\cdots+\alpha_{r}=n-r\\\alpha_0,\ldots,\alpha_{r}\geq 0}}U_s^{\alpha_0}W_s^{j_1}\cdots W_s^{j_{r}}U_s^{\alpha_{r}}\Bigg]\\
  	&+\sum_{k=1}^{q+1}\sum_{\substack{j_1+\cdots+j_{q+1}=q+1\\j_a\geq 1, j_k=1}}\dfrac{q!}{j_1!\cdots j_{k-1}!(j_k-1)!\cdots j_{q+1}!}\\
  	&\hspace{1.2in}\times\Bigg[\sum_{\substack{\alpha_0+\cdots+\alpha_{q+1}=n-(q+1)\\\alpha_0,\ldots,\alpha_{q+1}\geq 0}}U_s^{\alpha_0}W_s^{j_1}\cdots W_s^{j_{q+1}}U_s^{\alpha_{q+1}}\Bigg]\\
  	=&\sum_{\substack{\alpha_0+\alpha_1=n-1\\\alpha_0,\alpha_1\geq 0}}U_s^{\alpha_0}W_s^{q+1}U_s^{\alpha_1}\\
  	&+\sum_{r=2}^{q}~\sum_{k=1}^{r}~\sum_{\substack{j_1+\cdots+j_{r}=q+1\\j_1\ldots j_r\geq 1}}\dfrac{q!}{j_1!\cdots j_{k-1}!(j_k-1)!j_{k+1}\cdots j_r!}\\
  	&\hspace{1.2in}\times\Bigg[\sum_{\substack{\alpha_0+\cdots+\alpha_{r}=n-r\\\alpha_0,\ldots,\alpha_{r}\geq 0}}U_s^{\alpha_0}W_s^{j_1}\cdots W_s^{j_{r}}U_s^{\alpha_{r}}\Bigg]\\
  	&+(q+1)!\sum_{\substack{\alpha_0+\cdots+\alpha_{q+1}=n-(q+1)\\\alpha_0,\ldots,\alpha_{q+1}\geq 0}}U_s^{\alpha_0}W_s^{1}\cdots W_s^{1}U_s^{\alpha_{q+1}}.
  \end{align*}
  Since 
  $$\sum_{k=1}^r\dfrac{q!}{j_1!\cdots j_{k-1}!(j_k-1)!j_{k+1}\cdots j_r!}=q!~\dfrac{(j_1+\cdots +j_r)}{j_1!\cdots j_r!}=\dfrac{(q+1)!}{j_1!\cdots j_r!}, $$ 
  it follows that 
  \begin{align*}
  	K_1+K_2=&\sum_{\substack{\alpha_0+\alpha_1=n-1\\\alpha_0,\alpha_1\geq 0}}U_s^{\alpha_0}W_s^{q+1}C_s^{\alpha_1}+\sum_{r=2}^{q}~\sum_{\substack{j_1+\cdots+j_{r}=q+1\\j_1\ldots j_r\geq 1}}\dfrac{(q+1)!}{j_1!\cdots j_r!}\\
  	&\hspace{2.2in}\times\Bigg[\sum_{\substack{\alpha_0+\cdots+\alpha_{r}=n-r\\\alpha_0,\ldots,\alpha_{r}\geq 0}}U_s^{\alpha_0}W_s^{j_1}\cdots W_s^{j_{r}}U_s^{\alpha_{r}}\Bigg]\\
  	&+(q+1)!\sum_{\substack{\alpha_0+\cdots+\alpha_{q+1}=n-(q+1)\\\alpha_0,\ldots,\alpha_{m+1}\geq 0}}U_s^{\alpha_0}W_s^{1}\cdots W_s^{1}U_s^{\alpha_{q+1}}\\
  	=&\sum_{r=1}^{q+1}~\sum_{\substack{j_1+\cdots+j_{r}=q+1\\j_1,\ldots,l_{r}\geq 1}}~\dfrac{(q+1)!}{j_1!\cdots j_r!}~\Bigg[\sum_{\substack{\alpha_0+\alpha_1+\cdots+\alpha_r=n-r\\\alpha_0,\alpha_1,\ldots,\alpha_r\geq 0}}U_s^{\alpha_0}W_s^{j_1}U_s^{\alpha_1}\cdots W_s^{j_{r}}U_s^{\alpha_r}\Bigg].
  \end{align*}
  Therefore the identity \eqref{eq2} is true for $k=q+1$ and hence by principle of mathematical induction \eqref{eq2} is true for each $q\in\mathbb{N}$. This completes the proof.
  
\end{proof}
Now we are in a position to state and prove our main result in this section.
\begin{thm}\label{thconuni}
	Let $T$ and $V$ be two contractions in $\mathscr{H}$ such that
	\vspace{0.1in}
	
		$(i)$ $V^*V=VV^*=I$, and $\dim(\ker T)=\dim(\ker T^*)$,
		\vspace{0.1in}
		
		$(ii)$ $T-V\in\bnh$, and $(I-T^*T)^{1/2}\in\bnh$.
	\vspace{0.1in}
	
	Let $T=V_T |T|$ be the polar decomposition of $T$, where $V_T$ is a partial isometry on $\mathscr{H}$ and $|T|=(T^*T)^{1/2}$. Set
 $\mathcal{L}:=\begin{bmatrix}
		TV^*&-D_{T^*}V_T\Big|_{\DT}\\
		D_T V^*&T^*V_T\Big|_{\DT}\\
	\end{bmatrix}:\begin{bmatrix}
	\mathscr{H}\\\\
	\mathcal{D}_T
\end{bmatrix}\longrightarrow \begin{bmatrix}
\mathscr{H}\\\\
\mathcal{D}_T
\end{bmatrix}$. 
Then $\mathcal{L}$ is a unitary operator on $\mathscr{H}\oplus\DT$ and, hence there exists a unique self-adjoint operator $L\in\mathcal{B}_n(\mathscr{H}\oplus\DT)$ with $\sigma(L)\subseteq(-\pi,\pi]$ such that $\mathcal{L}=e^{iL}$. Furthermore, if we denote $V_s:=\PH e^{isL}V, s\in [0,1]$, then for $\phi\in \mathcal{F}_n(\mathbb{T})$, 
	\begin{align*}
		\left\{\phi(T)-\phi(V)-\sum_{k=1}^{n-1}\dfrac{1}{k!}\dfrac{d^k}{ds^k}\Big|_{s=0}\phi(V_s) \right\}\in\boh,
	\end{align*} 
and there exists an $L^1(\mathbb{T})$-function  $\xi_n$ 
(unique up to additive constant) depend only on $n,T$ and $V$ such that 
	\begin{align*}
		\textup{Tr}\left\{\phi(T)-\phi(V)-\sum_{k=1}^{n-1}\dfrac{1}{k!}\dfrac{d^k}{ds^k}\Big|_{s=0}\phi(V_s) \right\}=\int_{0}^{2\pi}\dfrac{d^n}{dt^n}\{\phi(e^{it})\}\xi_n(t)dt.
	\end{align*}
\end{thm}
\begin{proof}
 Since $T$ is a contraction on $\mathscr{H}$, then we can dilate $T$ to a unitary operator on the larger space containing  $\mathscr{H}$. Let $U_T$ be the corresponding minimal unitary dilation of $T$ on $\mathscr{F}=\mathbf{H}^2_{\mathcal{D}_{T^*}}(\mathbb{D})\oplus\mathscr{H}\oplus\mathbf{H}^2_{\mathcal{D}_{T}}(\mathbb{D})$ and hence as in \eqref{dilationmat} we have the following block matrix representation of $U_T$ on $\mathscr{F}$:
 \begin{align}\label{dilationmatalt}
 	U_T:=\begin{bmatrix}
 		S_{\DTs}^*&0&0\\
 		D_{T^*}P_{\DTs}&T&0\\
 		-T^*P_{\DTs}&D_{T}&S_{\mathcal{D}_T}
 	\end{bmatrix}:\begin{bmatrix}
 		\HDTs\\
 		\mathscr{H}\\
 		\HDT
 	\end{bmatrix}\longrightarrow \begin{bmatrix}
 		\HDTs\\
 		\mathscr{H}\\
 		\HDT
 	\end{bmatrix}.
 \end{align}
Now onward we denote $U_1:=U_T$. On the other hand we note that  $T=V_T |T|$ is the polar decomposition of $T$, where $|T|=(T^*T)^{1/2}$ and $V_T$ is an isometry from $\overline{Ran(T^*)}$ onto $\overline{Ran(T)}$. Therefore by using the hypothesis $\dim(\ker T)=\dim(\ker T^*)$, we can extend $V_T$ to a unitary operator on the full space $\mathscr{H}$. Moreover, as in \cite{Mor}, it is easy to observe that
\begin{align}\label{eq1}
	V_TD_T=D_{T^*}V_T,~(1-|T|)=(1+|T|)^{-1}(1-T^*T), \text{ and } V_T-T=V_T(1-|T|).
\end{align} 
Our next aim is to extend the unitary operator $V:\mathscr{H}\to \mathscr{H}$ to the larger space $\mathscr{F}$ as a unitary operator. To that aim, we set an operator $U_0:\mathscr{F}\to \mathscr{F}$ whose block matrix representation on the space $\mathscr{F}$ is the following:
	\begin{align}\label{dilationmatalt2}
	U_0:=\begin{bmatrix}
		S^*_{\DTs}&0&0\\
		0&V&0\\
		-V_{T}^*P_{\DTs}&0&S_{\mathcal{D}_T}
	\end{bmatrix}:\begin{bmatrix}
	\HDTs\\
	\mathscr{H}\\
	\HDT
\end{bmatrix}\longrightarrow \begin{bmatrix}
\HDTs\\
\mathscr{H}\\
\HDT
\end{bmatrix}.
\end{align}
Then using the relations $S_{\mathcal{D}_{T^*}}S_{\mathcal{D}_{T^*}}^*=I-P_{\mathcal{D}_{T^*}}$, $S_{\mathcal{D}_{T}}S_{\mathcal{D}_{T}}^*=I-P_{\mathcal{D}_{T}}$, $V_TP_{\mathcal{D}_T}=P_{\mathcal{D}_{T^*}}V_T$, $P_{\mathcal{D}_{T^*}}S_{\mathcal{D}_{T^*}}=0$, and $P_{\mathcal{D}_{T}}S_{\mathcal{D}_{T}}=0$, it is easy to check that $U_0$ is a unitary operator on $\mathscr{F}$, and $U_0$ is also an extension of $V$. Now the block matrix representation of $U_1-U_0$ on $\mathscr{F}$ is the following:
\begin{equation}\label{dilationmatalt1}
	U_1-U_0:= \begin{bmatrix}
	0&0&0\\
	D_{T^*}P_{\DTs}&T-V&0\\
	(V_{T}^*-T^*)P_{\DTs}&D_T&0
\end{bmatrix},
\end{equation}
which by using the relations listed in \eqref{eq1} together with the hypothesis $(ii)$ implies that each non-zero component of \eqref{dilationmatalt1} are in Schatten-n class and hence $U_1-U_0\in\bnf$. Now we have a pair $(U_1,U_0)$ of unitary operators on $\mathscr{F}$ and note that the block matrix representation of $U_1U_0^*$ with respect to the decomposition $\mathscr{F}=\HDTs\oplus(\mathscr{H}\oplus\DT)\oplus S_{\mathcal{D}_T}\HDT$ is the following: 
	\begin{align}\label{dilationmatalt8}
\nonumber	U_1U_0^*:=&\begin{bmatrix}
		I&0&0\\
		0&TV^*&-D_{T^*}V_T\Big|_{\DT}\\
		0&D_T V^*&T^*V_T\Big|_{\DT}+S_{\mathcal{D}_{T}}S_{\mathcal{D}_{T}}^*
	\end{bmatrix}:\begin{bmatrix}
		\HDTs\\
		\mathscr{H}\\
		\HDT
	\end{bmatrix}\to 
	\begin{bmatrix}
		\HDTs\\
		\mathscr{H}\\
		\HDT
	\end{bmatrix}\\
\nonumber	=&\begin{bmatrix}
		I&0&0&0\\
		0&TV^*&-D_{T^*}V_T\Big|_{\DT}&0\\
		0&D_T V^*&T^*V_T\Big|_{\DT}&0\\
		0&0&0&I
	\end{bmatrix}:\begin{bmatrix}
		\HDTs\\
		\mathscr{H}\\
		\DT\\
		S_{\DT}\HDT
	\end{bmatrix}\to 
	\begin{bmatrix}
		\HDTs\\
		\mathscr{H}\\
		\DT\\
		S_{\DT}\HDT
	\end{bmatrix}\\
	=&\begin{bmatrix}
		I&0&0\\
		0&\mathcal{L}&0\\
		0&0&I
	\end{bmatrix}:\begin{bmatrix}
		\HDTs\\
		\mathscr{H}\oplus\DT\\
		S_{\DT}\HDT
	\end{bmatrix}\to 
	\begin{bmatrix}
		\HDTs\\
		\mathscr{H}\oplus\DT\\
		S_{\DT}\HDT
	\end{bmatrix},
\end{align} 
where  we set $\mathcal{L}:=\begin{bmatrix}
	TV^*&-D_{T^*}V_T\Big|_{\DT}\\
	D_T V^*&T^*V_T\Big|_{\DT}\\
\end{bmatrix} :\begin{bmatrix}
\mathscr{H}\\\\
\DT
\end{bmatrix} \to  \begin{bmatrix}
\mathscr{H}\\\\
\DT
\end{bmatrix}$
 and $\mathcal{L}$ is a unitary operator on $(\mathscr{H}\oplus\DT)$. Therefore there exists a self-adjoint operator $L:\mathscr{H}\oplus\DT\to\mathscr{H}\oplus\DT$ with $\sigma(L)\subset (-\pi,\pi]$ such that $\mathcal{L}=e^{iL}$. 
Now if we set
 \begin{equation}\label{dilationmatalt3}
 A: =\begin{bmatrix}
	0&0&0\\
	0&L&0\\
	0&0&0
\end{bmatrix}:\begin{bmatrix}
	\HDTs\\
	\mathscr{H}\oplus\DT\\
	S_{\DT}\HDT
\end{bmatrix}\to 
\begin{bmatrix}
	\HDTs\\
	\mathscr{H}\oplus\DT\\
	S_{\DT}\HDT
\end{bmatrix},
\end{equation} 
then $A$ is a self-adjoint operator on $\mathscr{F}=\HDTs\oplus(\mathscr{H}\oplus\DT)\oplus S_{\mathcal{D}_T}\HDT$ such that $\sigma(A)\subset (-\pi,\pi]$, $U_1U_0^*=e^{iA}$ and hence  $U_1=e^{iA}U_0$.
Moreover, the equality $e^{iA}-I=(U_1-U_0)U_0^*$
together with the fact that $U_1-U_0\in \mathcal{B}_n(\mathscr{F})$ immediately implies $e^{iA}-I\in \mathcal{B}_n(\mathscr{F})$. Finally, by applying spectral theorem of the self-adjoint operator $A$ and using the inequality  $|x|\leq \dfrac{\pi}{2}|e^{ix}-1|,~x\in(-\pi,\pi],$ we conclude that  
 $A\in\bnf$ and hence $L\in\mathcal{B}_n(\mathscr{H}\oplus\DT)$. Therefore we have a pair $(U_1,U_0)$ of unitary operators on $\mathscr{F}$ satisfying the hypothesis of Theorem~\ref{th1} and hence for any $\phi\in \mathcal{F}_n(\mathbb{T})$,
\begin{equation}\label{constat1}
	\left\{\phi(U_1)-\phi(U_0)-\sum_{k=1}^{n-1}\dfrac{1}{k!}\dfrac{d^k}{ds^k}\Big|_{s=0}\phi(U_s) \right\}\in \bnh,
\end{equation}
and there exists an $L^1([0, 2\pi])$-function $\eta_n=\eta_{n,U_0,A}$ (unique up to an additive constant) such that 
\begin{align}\label{constat2}
	\textup{Tr}\left\{\phi(U_1)-\phi(U_0)-\sum_{k=1}^{n-1}\dfrac{1}{k!}\dfrac{d^k}{ds^k}\Big|_{s=0}\phi(U_s) \right\}=\int_{0}^{2\pi}\dfrac{d^n}{dt^n}\big\{\phi(e^{it})\big\}\eta_n(t)dt,
\end{align}
where $U_s=e^{isA}U_0,~s\in [0,1]$. Our next aim is to show that for $\phi\in \mathcal{F}_n(\mathbb{T})$,
\begin{align}\label{constat3}
	\textup{Tr}\left\{\phi(T)-\phi(V)-\sum_{k=1}^{n-1}\dfrac{1}{k!}\dfrac{d^k}{ds^k}\Big|_{s=0}\phi(V_s) \right\}=\textup{Tr}\left\{\phi(U_1)-\phi(U_0)-\sum_{k=1}^{n-1}\dfrac{1}{k!}\dfrac{d^k}{ds^k}\Big|_{s=0}\phi(U_s) \right\},
\end{align}
where \begin{equation}\label{CTeq1}
	V_s: =\PH e^{isA} U_0\Big|_\mathscr{H}=\PH e^{isL} V, ~s\in [0,1].
\end{equation}
To this end, it is enough to deal with the monomials, that is functions like $\phi_q(z)=z^q,~ z\in \mathbb{T}$ and $q\in \mathbb{Z}$. Now if $q\in \mathbb{N}$, then by using Lemma \ref{lm1}, we conclude for $1\leq k\leq n-1$ that
\begin{align}\label{eq4}
	\dfrac{d^k}{ds^k}\Big|_{s=0}\{\phi_q(U_s)\}=&\sum_{r=1}^{k}~\sum_{\substack{\alpha_0+\cdots+\alpha_r=q-r\\\alpha_0,\ldots,\alpha_r\geq 0}}~\sum_{\substack{l_1+\cdots+l_{r}=k\\l_1,\ldots,l_{r}\geq 1}}\dfrac{k!}{l_1!\cdots l_r!}U_0^{\alpha_0}((iA)^{l_1}U_0)U_0^{\alpha_1}\cdots ((iA)^{l_{r}}U_0)U_0^{\alpha_r},
\end{align}
and 
\begin{align}\label{eq5}
	\dfrac{d^k}{ds^k}\Big|_{s=0}\{\phi_q(V_s)\}=&\sum_{r=1}^{k}~\sum_{\substack{\alpha_0+\cdots+\alpha_r=q-r\\\alpha_0,\ldots,\alpha_r\geq 0}}~\sum_{\substack{l_1+\cdots+l_{r}=k\\l_1,\ldots,l_{r}\geq 1}}\dfrac{k!}{l_1!\cdots l_r!}V^{\alpha_0}\PH ((iL)^{l_1}V)V^{\alpha_1}\cdots \PH((iL)^{l_{r}}V)V^{\alpha_r}.
\end{align}
Next we denote $X_0:=U_1^n-U_0^n$, $X_r:=U_0^{\alpha_0}((iA)^{l_1}U_0)U_0^{\alpha_1}\cdots ((iA)^{l_{r}}U_0)U_0^{\alpha_r}$, where $\alpha_j\geq 0$ for $0\leq j\leq r$, and $l_{j'}\geq 1$ for $1\leq j'\leq r$ , and $r\geq 1$. Then by analyzing the block matrix representations
\eqref{dilationmatalt}, \eqref{dilationmatalt2} and \eqref{dilationmatalt3} of $U_1$, $U_0$ and $A$ respectively we conclude 
\begin{align}
	\label{eq3}\PH X_0\Big|_{\mathscr{H}}=&\PH (U_1^n-U_0^n)\Big|_{\mathscr{H}}=T^n-V^n,\\
	\nonumber\PH X_r\Big|_{\mathscr{H}}=&\PH U_0^{\alpha_0}((iA)^{l_1}U_0)U_0^{\alpha_1}\cdots ((iA)^{l_{r}}U_0)U_0^{\alpha_r}\Big|_{\mathscr{H}}\\
	\nonumber=&\PH U_0^{\alpha_0}P_{\mathscr{H}\oplus\mathcal{D}_{T}}((iA)^{l_1}P_{\mathscr{H}\oplus\mathcal{D}_{T}}U_0)U_0^{\alpha_1}P_{\mathscr{H}\oplus\mathcal{D}_{T}}\cdots U_0^{\alpha_{r-1}}P_{\mathscr{H}\oplus\mathcal{D}_{T}}((iA)^{l_{r}}P_{\mathscr{H}\oplus\mathcal{D}_{T}}U_0)U_0^{\alpha_r}\Big|_{\mathscr{H}}\\
	\label{eq6}=& V^{\alpha_0}\PH ((iL)^{l_1}V)V^{\alpha_1}\cdots V^{\alpha_{r-1}}\PH ((iL)^{l_{r}}V)V^{\alpha_r}.
\end{align}
Therefore combining equations \eqref{eq4},\eqref{eq5},\eqref{eq3}, and \eqref{eq6} we get
\begin{align}\label{dilationmatalt4}
	\phi_q(T)-\phi_q(V)-\sum_{k=1}^{n-1}\dfrac{1}{k!}\dfrac{d^k}{ds^k}\Big|_{s=0}\phi_q(V_s) =\PH\left(\phi_q(U_1)-\phi_q(U_0)-\sum_{k=1}^{n-1}\dfrac{1}{k!}\dfrac{d^k}{ds^k}\Big|_{s=0}\phi_q(U_s)\right)\Bigg|_{\mathscr{H}},
\end{align}
for all $q\in \mathbb{N}$. Now for $q\in \mathbb{Z}$, $q<0$, recall that $T^q={T^*}^{-q}$ for a given contraction $T$ as in \eqref{INTeq1}. Therefore using the facts
$\dfrac{d^k}{ds^k}\Big|_{s=0}\{\phi_q(U_s)\} =\left(\dfrac{d^k}{ds^k}\Big|_{s=0}\{\phi_{-q}(U_s)\}\right)^*$ and 
$\dfrac{d^k}{ds^k}\Big|_{s=0}\{\phi_q(V_s)\} =\left(\dfrac{d^k}{ds^k}\Big|_{s=0}\{\phi_{-q}(V_s)\}\right)^*$, and again by applying Lemma \ref{lm1} together with the same analysis as above we also get for $q\in \mathbb{Z},~q<0$ that
\begin{align}\label{dilationmatalt5}
	\phi_q(T)-\phi_q(V)-\sum_{k=1}^{n-1}\dfrac{1}{k!}\dfrac{d^k}{ds^k}\Big|_{s=0}\phi_q(V_s) =\PH\left(\phi_q(U_1)-\phi_q(U_0)-\sum_{k=1}^{n-1}\dfrac{1}{k!}\dfrac{d^k}{ds^k}\Big|_{s=0}\phi_q(U_s)\right)\Bigg|_{\mathscr{H}}.
\end{align}
Therefore using \eqref{constat1}, \eqref{dilationmatalt4} and \eqref{dilationmatalt5} we conclude 
$\left\{\phi_q(T)-\phi_q(V)-\sum\limits_{k=1}^{n-1}\dfrac{1}{k!}\dfrac{d^k}{ds^k}\Big|_{s=0}\phi_q(V_s)\right\}$
is a trace class operator and 
\begin{equation} \label{dilationmatalt6}
	\begin{split}
	&\textup{Tr}\left\{\phi_q(T)-\phi_q(V)-\sum_{k=1}^{n-1}\dfrac{1}{k!}\dfrac{d^k}{ds^k}\Big|_{s=0}\phi_q(V_s)\right\}\\ &\hspace{1in}=\textup{Tr}\left\{\PH\left(\phi_q(U_1)-\phi_q(U_0)-\sum_{k=1}^{n-1}\dfrac{1}{k!}\dfrac{d^k}{ds^k}\Big|_{s=0}\phi_q(U_s)\right)\Bigg|_{\mathscr{H}}\right\},\forall q\in \mathbb{Z}.
	\end{split}
\end{equation}
On the other hand again by analyzing the structures of $U_1$, $U_0$, $U_1-U_0$ and $A$ as in \eqref{dilationmatalt}, \eqref{dilationmatalt2}, \eqref{dilationmatalt1} and \eqref{dilationmatalt3}  respectively we get
\begin{align*}
 P_{\mathscr{F}\ominus\mathscr{H}}X_0\Big|_{\mathscr{F}\ominus\mathscr{H}}=&\sum_{k=0}^{n-1}P_{\mathscr{F}\ominus\mathscr{H}}U_1^{n-k-1}(U_1-U_0)U_0^k\Big|_{\mathscr{F}\ominus\mathscr{H}}\\ =&\sum_{k=0}^{n-1}P_{\mathscr{F}\ominus\mathscr{H}}U_1^{n-k-1}(U_1-U_0)U_0^kP_{\HDTs} \Big|_{\mathscr{F}\ominus\mathscr{H}}\\
	=&\sum_{k=0}^{n-1}P_{\mathscr{F}\ominus\mathscr{H}}U_1^{n-k-1}P_{\mathscr{H}\oplus\HDT}(U_1-U_0)U_0^kP_{\HDTs} \Big|_{\mathscr{F}\ominus\mathscr{H}}\\
	=&\sum_{k=0}^{n-1}P_{\HDT}U_1^{n-k-1}P_{\mathscr{H}\oplus\HDT}(U_1-U_0)U_0^kP_{\HDTs} \Big|_{\mathscr{F}\ominus\mathscr{H}},
\end{align*}
and for $r\geq 1$,
\begin{align*}
	 P_{\mathscr{F}\ominus\mathscr{H}}X_r\Big|_{\mathscr{F}\ominus\mathscr{H}}=&P_{\mathscr{F}\ominus\mathscr{H}}U_0^{\alpha_0}((iA)^{l_1}U_0)U_0^{\alpha_1}\cdots ((iA)^{l_{r}}U_0)U_0^{\alpha_r}\Big|_{\mathscr{F}\ominus\mathscr{H}}\\
	=&P_{\mathscr{F}\ominus\mathscr{H}}U_0^{\alpha_0}P_{\mathscr{H}\oplus\mathcal{D}_{T}}((iA)^{l_1}P_{\mathscr{H}\oplus\mathcal{D}_{T}}U_0)U_0^{\alpha_1}\cdots P_{\mathscr{H}\oplus\mathcal{D}_{T}}((iA)^{l_{r}}P_{\mathscr{H}\oplus\mathcal{D}_{T}}U_0)U_0^{\alpha_r}\Big|_{\mathscr{F}\ominus\mathscr{H}}\\
	=&P_{\HDT}U_0^{\alpha_0}P_{\mathscr{H}\oplus\mathcal{D}_{T}}((iA)^{l_1}P_{\mathscr{H}\oplus\mathcal{D}_{T}}U_0)U_0^{\alpha_1}\\
	&\hspace*{1in}\times\cdots\times P_{\mathscr{H}\oplus\mathcal{D}_{T}}((iA)^{l_{r}}P_{\mathscr{H}\oplus\mathcal{D}_{T}}U_0)U_0^{\alpha_r}P_{\HDTs}\Big|_{\mathscr{F}\ominus\mathscr{H}},
\end{align*}
which implies that the operator $P_{\mathscr{F}\ominus\mathscr{H}}\Bigg(\phi_q(U_1)-\phi_q(U_0)-\sum\limits_{k=1}^{n-1}\dfrac{1}{k!}\dfrac{d^k}{ds^k}\Big|_{s=0}\phi_q(U_s)\Bigg)\Bigg|_{\mathscr{F}\ominus\mathscr{H}}$ maps $\HDTs\oplus \textbf{0}\oplus \textbf{0}$ to $\textbf{0}\oplus \textbf{0}\oplus\HDT$ for $q\in \mathbb{N}$ and the operator  $P_{\mathscr{F}\ominus\mathscr{H}}\Bigg(\phi_q(U_1)-\phi_q(U_0)$ $-\sum\limits_{k=1}^{n-1}\dfrac{1}{k!}\dfrac{d^k}{ds^k}\Big|_{s=0}\phi_q(U_s)\Bigg)\Bigg|_{\mathscr{F}\ominus\mathscr{H}}$ maps $ \textbf{0}\oplus \textbf{0}\oplus\HDT$ to $\HDTs\oplus \textbf{0}\oplus \textbf{0}$ for $q\in \mathbb{Z},~q<0$. These observations immediately yields that 
\begin{align}\label{dilationmatalt7}
	\textup{Tr}\left\{P_{\mathscr{F}\ominus\mathscr{H}}\left(\phi_q(U_1)-\phi_q(U_0)-\sum_{k=1}^{n-1}\dfrac{1}{k!}\dfrac{d^k}{ds^k}\Big|_{s=0}\phi_q(U_s)\right)\Bigg|_{\mathscr{F}\ominus\mathscr{H}} \right\}=0, ~~\forall q\in \mathbb{Z}.
\end{align}
Therefore combining equations \eqref{dilationmatalt6} and \eqref{dilationmatalt7} we get
\begin{equation*}
	\textup{Tr}\left\{\phi_q(T)-\phi_q(V)-\sum_{k=1}^{n-1}\dfrac{1}{k!}\dfrac{d^k}{ds^k}\Big|_{s=0}\phi_q(V_s)\right\} =\textup{Tr}\left\{\phi_q(U_1)-\phi_q(U_0)-\sum_{k=1}^{n-1}\dfrac{1}{k!}\dfrac{d^k}{ds^k}\Big|_{s=0}\phi_q(U_s)\right\}
\end{equation*}
for all $ q\in \mathbb{Z}$ and hence \eqref{constat3} follows. Finally the conclusion of the theorem follows by combining equations \eqref{constat2} and \eqref{constat3}. This completes the proof.  
\end{proof}

\section{Higher order Trace formula for pair of contractions }\label{concon}
In the previous section, we discuss the trace formula for pairs of contractions $(T, V)$ with the assumption that $V$ is unitary. In this section, we remove the assumption on $V$. In other words, we prove the trace formula for pairs of contractions $(T_0,T_1)$ on $\hil$. The technique involved here is standard and similar to the idea mentioned in \cite{Mor} with an appropriate modification, that means first we dilate $(T_0, T_1)$ to a pair of contractions $(T, V)$ with $V$ is a unitary operator on the bigger space $\mathscr{F}$ containing $\hil$ as a subspace and then use the existing trace formula for the pair  $(T, V)$ obtained in our last section to get the required trace formula in this section. The following is the main result in this section.

\begin{thm}\label{thconcon}
	Let $T_0$ and $T_1$ be two contractions in $\mathscr{H}$ such that
	\vspace{0.1in}
	
	$(i)$~$\dim(\ker T_0)=\dim(\ker T_0^*)$, and ~$\dim(\ker T_1)=\dim(\ker T_1^*)$,
	\vspace{0.1in}
	
	$(ii)$~$T_1-T_0\in\bnh$, and ~$(I-T_j^*T_j)^{1/2}\in\bnh$ for $j=0,1$.
\vspace{0.1in}

Let $T_j=V_{T_j} |T_j|$ be the polar decomposition of $T_j$, where $V_{T_j}$ is a partial isometry on $\mathscr{H}$ and $|T_j|=(T_j^*T_j)^{1/2}$ for $j=0,1$. Set 
\begin{align*}
	\mathcal{M}:=\begin{bmatrix}
		I&0&0&0\\
		0&T_1T_0^*&T_1D_{T_0}P_{\mathcal{D}_{T_0}}&-D_{T_1^*}V_{T_1}\\
		0&-V_{T_0}^*D_{T_0^*}& |T_0|P_{\mathcal{D}_{T_0}}+(I-P_{\mathcal{D}_{T_0}})&0\\
		0&D_{T_1}T_0^*&D_{T_1}D_{T_0}P_{\mathcal{D}_{T_0}}&T_1^*V_{T_1}
	\end{bmatrix}:
	\begin{bmatrix}
		\HDTzeros\\
		\mathscr{H}\\
		\HDTzero\\
		\mathcal{D}_{T_1}	
	\end{bmatrix}\to \begin{bmatrix}
		\HDTzeros\\
		\mathscr{H}\\
		\HDTzero\\
		\mathcal{D}_{T_1}	
	\end{bmatrix}.
\end{align*}
Then $\mathcal{M}$ is a unitary operator on $\HDTzeros\oplus\mathscr{H}\oplus\HDTzero\oplus\mathcal{D}_{T_1}=\mathscr{F}\oplus\mathcal{D}_{T_1}$ and hence there exists a unique self-adjoint operator 
 $M\in\mathcal{B}_n(\mathscr{F}\oplus\mathcal{D}_{T_1})$ with $\sigma(M)\subseteq(-\pi,\pi]$ such that $\mathcal{M}=e^{iM}$. Furthermore, if we denote 
 \begin{equation}\label{PCeqpath}
 	T_s= \PH e^{isM}\begin{bmatrix}
 		0\\
 		T_0\\
 		D_{T_0}\\
 		0
 	\end{bmatrix}:\mathscr{H} \to \mathscr{H},~~ s\in [0,1],
 \end{equation}
then for $\phi\in \mathcal{F}_n(\mathbb{T})$, 
\begin{align*}
	\left\{\phi(T_1)-\phi(T_0)-\sum_{k=1}^{n-1}\dfrac{1}{k!}\dfrac{d^k}{ds^k}\Big|_{s=0}\phi(T_s) \right\}\in\boh,
\end{align*} 
and there exists an $L^1(\mathbb{T})$-function  $\xi_n$ 
(unique up to additive constant) depend only on $n,T_1$ and $T_0$ such that 
\begin{align*}
	\textup{Tr}\left\{\phi(T_1)-\phi(T_0)-\sum_{k=1}^{n-1}\dfrac{1}{k!}\dfrac{d^k}{ds^k}\Big|_{s=0}\phi(T_s) \right\}=\int_{0}^{2\pi}\dfrac{d^n}{dt^n}\{\phi(e^{it})\}\xi_n(t)dt.
\end{align*}
\end{thm}
\begin{proof}
Since $T_0$ is a contraction on $\mathscr{H}$, then we dilate $T_0$ to a unitary operator on the larger space $\mathscr{F}$ containing  $\mathscr{H}$. Let $U_{T_0}$ be the corresponding minimal unitary dilation of $T_0$ on $\mathscr{F}=\mathbf{H}^2_{\mathcal{D}_{T_0^*}}(\mathbb{D})\oplus\mathscr{H}\oplus\mathbf{H}^2_{\mathcal{D}_{T_0}}(\mathbb{D})$ and hence as in \eqref{dilationmat} we have the following block matrix representation of $U_{T_0}$ on $\mathscr{F}$:
\begin{align}\label{vrep}
	U_{T_0}=\begin{bmatrix}
		S_{\mathcal{D}_{T_0^*}}^*&0&0\\
		D_{T_{0}^*}P_{\mathcal{D}_{T_0^*}}&T_0&0\\
		-T_{0}^*P_{\mathcal{D}_{T_0^*}}&D_{T_0}&S_{\mathcal{D}_{T_0}}
	\end{bmatrix}:\begin{bmatrix}
		\HDTzeros\\
		\mathscr{H}\\
		\HDTzero
	\end{bmatrix}\to 
	\begin{bmatrix}
		\HDTzeros\\
		\mathscr{H}\\
		\HDTzero
	\end{bmatrix}.
\end{align}
Now we set $V:=U_{T_0}$. On the other hand we note that  $T_0=V_{T_0} |T_0|$ is the polar decomposition of $T_0$, where $|T_0|=(T_0^*T_0)^{1/2}$ and $V_{T_0}$ is an isometry from $\overline{Ran(T_0^*)}$ onto $\overline{Ran(T_0)}$. Therefore by using the hypothesis $\dim(\ker T_0)=\dim(\ker T_0^*)$, we  extend $V_{T_0}$ to a unitary operator on the full space $\mathscr{H}$. Moreover, as in \cite{Mor}, it is easy to observe that
\begin{align}\label{PCeq1}
	V_{T_0}D_{T_0}=D_{T_0^*}V_{T_0},~(1-|T_0|)=(1+|T_0|)^{-1}(1-T_0^*T_0), \text{ and } V_{T_0}-T_0=V_{T_0}(1-|T_0|).
\end{align} 
Our next aim is to extend the contraction $T_1:\mathscr{H}\to \mathscr{H}$ to the larger space $\mathscr{F}$ as a contractive operator. To that aim, we set an operator $T:\mathscr{F}\to \mathscr{F}$ whose block matrix representation on the space $\mathscr{F}$ is the following:
\begin{align}\label{trep}
	T:= \begin{bmatrix}
		S_{\mathcal{D}_{T_0^*}}^*&0&0\\
		0&T_1&0\\
		-V_{T_0}^*P_{\mathcal{D}_{T_0^*}}&0&S_{\mathcal{D}_{T_0}}
	\end{bmatrix}:\begin{bmatrix}
		\HDTzeros\\
		\mathscr{H}\\
		\HDTzero
	\end{bmatrix}\to 
	\begin{bmatrix}
		\HDTzeros\\
		\mathscr{H}\\
		\HDTzero
	\end{bmatrix}.
\end{align}
Then using the relations $S_{\mathcal{D}_{T_0^*}}S_{\mathcal{D}_{T_0^*}}^*=I-P_{\mathcal{D}_{T_0^*}}$, $S_{\mathcal{D}_{T_0}}S_{\mathcal{D}_{T_0}}^*=I-P_{\mathcal{D}_{T_0}}$, $V_{T_0}P_{\mathcal{D}_{T_0}}=P_{\mathcal{D}_{T_0^*}}V_{T_0}$, $P_{\mathcal{D}_{T_0^*}}S_{\mathcal{D}_{T_0^*}}=0$, and $P_{\mathcal{D}_{T_0}}S_{\mathcal{D}_{T_0}}=0$, it is easy to check that $T$ is a contraction on $\mathscr{F}$, and $T$ is also an extension of $T_1$. Moreover, using the above relations we also conclude
	\begin{align}
	\label{q2}|T|=&\begin{bmatrix}
		I&0&0\\
		0&|T_1|&0\\
		0&0&I
	\end{bmatrix}:
	\begin{bmatrix}
		\HDTzeros\\
		\mathscr{H}\\
		\HDTzero
	\end{bmatrix}\to 
	\begin{bmatrix}
		\HDTzeros\\
		\mathscr{H}\\
		\HDTzero
	\end{bmatrix},\\
	\label{q3}|T^*|=&\begin{bmatrix}
		I&0&0\\
		0&|T_1^*|&0\\
		0&0&I
	\end{bmatrix}:
	\begin{bmatrix}
		\HDTzeros\\
		\mathscr{H}\\
		\HDTzero
	\end{bmatrix}\to 
	\begin{bmatrix}
		\HDTzeros\\
		\mathscr{H}\\
		\HDTzero
	\end{bmatrix},\\
	\label{q4}	D_T=&\begin{bmatrix}
		0&0&0\\
		0&D_{T_1}&0\\
		0&0&0
	\end{bmatrix}:\begin{bmatrix}
		\HDTzeros\\
		\mathscr{H}\\
		\HDTzero
	\end{bmatrix}\to 
	\begin{bmatrix}
		\HDTzeros\\
		\mathscr{H}\\
		\HDTzero
	\end{bmatrix}.
\end{align}
Now we have a pair  $(T,V)$ of contractions on $\mathscr{F}=\mathbf{H}^2_{\mathcal{D}_{T_0^*}}(\mathbb{D})\oplus\mathscr{H}\oplus\mathbf{H}^2_{\mathcal{D}_{T_0}}(\mathbb{D})$ such that $V$ is a unitary operator. Next we show that the pair $(T,V)$ satisfies the conditions $(i)$ and $(ii)$ in Theorem~\ref{thconuni}. The block matrix representations \eqref{q2} and \eqref{q3} of $|T|$ and $|T^*|$ respectively implies that $\ker T=\ker T_1$ and $\ker T^*=\ker T_1^*$, and hence $\dim(\ker T)=\dim(\ker T^*)$ by using the hypothesis $(i)$. Furthermore, using the block matrix representations $\eqref{vrep}$ and $\eqref{trep}$ of $V$ and $T$ respectively we have the following block matrix representation $T-V$ on $\mathscr{F}$
\begin{equation}\label{PCeq2}
	T-V:= \begin{bmatrix}
		0&0&0\\
		-D_{T_0^*}P_{\mathcal{D}_{T_0^*}}&T_1-T_0&0\\
		(T_0^*-V_{T_0}^*)P_{\mathcal{D}_{T_0^*}}&-D_{T_0}&0
	\end{bmatrix},
\end{equation}
which by using the relations listed in \eqref{PCeq1} together with the hypothesis $(ii)$ and using the equation \eqref{q4} yield that each non-zero component of \eqref{PCeq2} are in Schatten-n class and hence $T-V\in\bnf$.
 Next we obtain the polar decomposition of $T$ from the polar decomposition of $T_1$. Indeed, note that $T_1=V_{T_1} |T_1|$ is the polar decomposition of $T_1$, where $|T_1|=(T_1^*T_1)^{1/2}$ and $V_{T_1}$ is an isometry from $\overline{Ran(T_1^*)}$ onto $\overline{Ran(T_1)}$. Then by using the hypothesis $\dim(\ker T_1)=\dim(\ker T_1^*)$, we  extend $V_{T_1}$ to a unitary operator on the full space $\mathscr{H}$. Next we set
\begin{align}\label{PCeq3}
	V_T:=&\begin{bmatrix}
		S_{\DTzeros}^* & 0 & 0\\
		0&V_{T_1}&0 & \\
		-V_{T_0}^*P_{\DTzeros}&0&S_{\DTzero}
	\end{bmatrix}:\begin{bmatrix}
		\HDTzeros\\
		\mathscr{H}\\
		\HDTzero
	\end{bmatrix}\to 
	\begin{bmatrix}
		\HDTzeros\\
		\mathscr{H}\\
		\HDTzero
	\end{bmatrix}.
\end{align}
Then using equations \eqref{trep}, \eqref{q2} and \eqref{PCeq3} we conclude that $T=V_T|T|$. Now, to apply Theorem~\ref{thconuni} in this context corresponding to the pair $(T,V)$ we need to find the contractive path as discussed earlier in \eqref{CTeq1} and hence we proceed in a similar way as in the proof of Theorem~\ref{thconuni} with an appropriate modification. Let $U_1$ and $U_0$ be the minimal unitary dilation and unitary extension of $T$ and $V$ respectively on the space $\mathcal{K}:=\HDTs\oplus\mathscr{F}\oplus\HDT$. Moreover, the block matrix representations of $U_1$ and $U_0$ on $\mathcal{K}$ are given below:
\begin{align*}
	U_1:=\begin{bmatrix}
		S_{\mathcal{D}_{T^*}}^*&0&0\\
		D_{T^*}P_{\mathcal{D}_{T^*}}&T&0\\
		-T^*P_{\mathcal{D}_{T^*}}&D_{T}&S_{\mathcal{D}_{T}}
	\end{bmatrix}:
	\begin{bmatrix}
		\HDTs\\
		\mathscr{F}\\
		\HDT
	\end{bmatrix}\to 
	\begin{bmatrix}
		\HDTs\\
		\mathscr{F}\\
		\HDT
	\end{bmatrix},
\end{align*}
and 
 \begin{align*}
	U_0:=\begin{bmatrix}
		S_{\mathcal{D}_{T^*}}^*&0&0\\
		0&V&0\\
		-V_{T}^*P_{\mathcal{D}_{T^*}}&0&S_{\mathcal{D}_{T}}
	\end{bmatrix}:\begin{bmatrix}
		\HDTs\\
		\mathscr{F}\\
		\HDT
	\end{bmatrix}\to 
	\begin{bmatrix}
		\HDTs\\
		\mathscr{F}\\
		\HDT
	\end{bmatrix}.
\end{align*}
Then using the fact that $T-V\in\bnf$ and using the relations listed in \eqref{PCeq1} with $T_0$ is replaced by $T$ we conclude that $U_1-U_0\in \mathcal{B}_n(\mathcal{K})$. Now we have a pair $(U_1,U_0)$ of unitary operators on $\mathcal{K}$ and as in \eqref{dilationmatalt8} the block matrix representation of $U_1U_0^*$ with respect to the decomposition $\mathcal{K}=\HDTs\oplus(\mathscr{F}\oplus\DT)\oplus S_{\mathcal{D}_T}\HDT$ is the following:
	\begin{equation}\label{q1}
	\begin{split}
		U_1U_0^*:= \begin{bmatrix}
			I&0&0\\
			0&\mathcal{M}&0\\
			0&0&I
		\end{bmatrix}:\begin{bmatrix}
			\HDTs\\
			\mathscr{F}\oplus\mathcal{D}_T\\
			S_{\mathcal{D}_{T}}\HDT
		\end{bmatrix}\to \begin{bmatrix}
			\HDTs\\
			\mathscr{F}\oplus\mathcal{D}_T\\
			S_{\mathcal{D}_{T}}\HDT,
		\end{bmatrix},
	\end{split}
\end{equation} 
where $\mathcal{M}=\begin{bmatrix}
	TV^*&-D_{T^*}V_T\big|_{\mathcal{D}_T}\\
	D_TV^*&T^*V_T\big|_{\mathcal{D}_T}\\
\end{bmatrix}:\begin{bmatrix}
	\mathscr{F}\\
	\mathcal{D}_T
\end{bmatrix}\to \begin{bmatrix}
	\mathscr{F}\\
	\mathcal{D}_T
\end{bmatrix}$ is a unitary operator. Therefore using the similar argument as in the proof of Theorem~\ref{thconuni} and using the fact that  $U_1-U_0\in \mathcal{B}_n(\mathcal{K})$, we conclude that there exists a self-adjoint operator $M:\mathscr{F}\oplus\DT\to\mathscr{F}\oplus\DT$ with $\sigma(M)\subseteq (-\pi,\pi]$ such that $M\in\mathcal{B}_n(\mathscr{F}\oplus\DT)$ and $\mathcal{M}=e^{iM}$. Moreover, if we set 
\begin{equation}\label{PCeq6}
	A: =\begin{bmatrix}
		0&0&0\\
		0&M&0\\
		0&0&0
	\end{bmatrix}:\begin{bmatrix}
		\HDTs\\
		\mathscr{F}\oplus\DT\\
		S_{\DT}\HDT
	\end{bmatrix}\to 
	\begin{bmatrix}
		\HDTs\\
		\mathscr{F}\oplus\DT\\
		S_{\DT}\HDT
	\end{bmatrix},
\end{equation} 
then $A$ is a self-adjoint operator on $\mathcal{K}=\HDTs\oplus(\mathscr{F}\oplus\DT)\oplus S_{\mathcal{D}_T}\HDT$ such that $\sigma(A)\subseteq (-\pi,\pi]$, $A\in \mathcal{B}_n(\mathcal{K})$,~$U_1U_0^*=e^{iA}$ and hence  $U_1=e^{iA}U_0$.
From the block matrix representation \eqref{q4} of $D_T$, it is easy to observe that $\mathcal{D}_T=0\oplus\mathcal{D}_{T_1}\oplus 0$. Keeping this information in mind we have the following block matrix representation of the operator $\mathcal{M}$ and it is necessary to get the contractive path in the sequel.
\begin{align}\label{PCeq11}
\nonumber	\mathcal{M}=&\begin{bmatrix}
		I&0&0&0\\
		0&T_1T_0^*&T_1D_{T_0}P_{\mathcal{D}_{T_0}}&-D_{T_1^*}V_{T_1}\\
		0&-V_{T_0}^*D_{T_0^*}& S_{\DTzero}S_{\DTzero}^*+V_{T_0}^*T_0P_{\mathcal{D}_{T_0}}&0\\
		0&D_{T_1}T_0^*&D_{T_1}D_{T_0}P_{\mathcal{D}_{T_0}}&T_1^*V_{T_1}
	\end{bmatrix}:
	\begin{bmatrix}
		\HDTzeros\\
		\mathscr{H}\\
		\HDTzero\\
		\mathcal{D}_{T_1}	
	\end{bmatrix}\to \begin{bmatrix}
		\HDTzeros\\
		\mathscr{H}\\
		\HDTzero\\
		\mathcal{D}_{T_1}	
	\end{bmatrix}\\
\nonumber	=&\begin{bmatrix}
		I&0&0&0\\
		0&T_1T_0^*&T_1D_{T_0}P_{\mathcal{D}_{T_0}}&-D_{T_1^*}V_{T_1}\\
		0&-V_{T_0}^*D_{T_0^*}& |T_0|P_{\mathcal{D}_{T_0}}+(I-P_{\mathcal{D}_{T_0}})&0\\
		0&D_{T_1}T_0^*&D_{T_1}D_{T_0}P_{\mathcal{D}_{T_0}}&T_1^*V_{T_1}
	\end{bmatrix}:
	\begin{bmatrix}
		\HDTzeros\\
		\mathscr{H}\\
		\HDTzero\\
		\mathcal{D}_{T_1}	
	\end{bmatrix}\to \begin{bmatrix}
		\HDTzeros\\
		\mathscr{H}\\
		\HDTzero\\
		\mathcal{D}_{T_1}	
	\end{bmatrix}\\
	=&\begin{bmatrix}
		I&0&0&0&0\\
		0&T_1T_0^*&T_1D_{T_0}&0&-D_{T_1^*}V_{T_1}\\
		0&-V_{T_0}^*D_{T_0^*}& |T_0|&0&0\\
		0&0&0&I&0\\
		0&D_{T_1}T_0^*&D_{T_1}D_{T_0}&0&T_1^*V_{T_1}
	\end{bmatrix}:
	\begin{bmatrix}
		\HDTzeros\\
		\mathscr{H}\\
		\mathcal{D}_{T_0}\\
		S_{\DTzero}\HDTzero\\
		\mathcal{D}_{T_1}	
	\end{bmatrix}\to \begin{bmatrix}
		\HDTzeros\\
		\mathscr{H}\\
		\mathcal{D}_{T_0}\\
		S_{\DTzero}\HDTzero\\
		\mathcal{D}_{T_1}	
	\end{bmatrix},
\end{align}
where at the second equality we use the relation $S_{\DTzero}S_{\DTzero}^*=I-P_{\mathcal{D}_{T_0}}$, and in the last equality we use the relation $V_{T_0}D_{T_0}=D_{T_0^*}V_{T_0}$. 
Therefore applying Theorem~\ref{thconuni} corresponding to the pair $(T,V)$ we conclude that  for $\phi\in \mathcal{F}_n(\mathbb{T})$, 
\begin{align}\label{PCeq4}
	\left\{\phi(T)-\phi(V)-\sum_{k=1}^{n-1}\dfrac{1}{k!}\dfrac{d^k}{ds^k}\Big|_{s=0}\phi(V_s) \right\}\in\boh,
\end{align} 
and there exists an $L^1(\mathbb{T})$-function  $\xi_n$ 
(unique up to additive constant) depend only on $n,T$ and $V$ such that 
\begin{align}\label{PCeq5}
	\textup{Tr}\left\{\phi(T)-\phi(V)-\sum_{k=1}^{n-1}\dfrac{1}{k!}\dfrac{d^k}{ds^k}\Big|_{s=0}\phi(V_s) \right\}=\int_{0}^{2\pi}\dfrac{d^n}{dt^n}\{\phi(e^{it})\}\xi_n(t)dt,
\end{align}
where $V_s=\PF e^{isM}V$, $s\in [0,1]$. Our next aim is to show that for $\phi\in \mathcal{F}_n(\mathbb{T})$,
\begin{align}\label{PCeq7}
	\textup{Tr}\left\{\phi(T_1)-\phi(T_0)-\sum_{k=1}^{n-1}\dfrac{1}{k!}\dfrac{d^k}{ds^k}\Big|_{s=0}\phi(T_s) \right\}=\textup{Tr}\left\{\phi(T)-\phi(V)-\sum_{k=1}^{n-1}\dfrac{1}{k!}\dfrac{d^k}{ds^k}\Big|_{s=0}\phi(V_s) \right\},
\end{align}
where \begin{equation}\label{PCeq8}
	T_s=P_{\mathscr{H}}V_s\Big|_{\mathscr{H}}=\PH e^{isB}\begin{bmatrix}
		0\\
		T_0\\
		D_{T_0}\\
		0
	\end{bmatrix}=\PH e^{isM}W,\quad \text{where} \quad W:=V\Big|_{\mathscr{H}}=\begin{bmatrix}
		0\\
		T_0\\
		D_{T_0}\\
		0
	\end{bmatrix},
\end{equation}
is a bounded  operator from $\hil$ to $\mathscr{F}\oplus\DT$. To proceed further,  it is enough to deal with the monomials, that is functions like $\phi_q(z)=z^q,~ z\in \mathbb{T}$ and $q\in \mathbb{Z}$. Now if $q\in \mathbb{N}$, then by using Lemma \ref{lm1}, we conclude for $1\leq k\leq n-1$ that
\begin{align}\label{PCeq9}
	\nonumber\dfrac{d^k}{ds^k}\Big|_{s=0}\{\phi_q(V_s)\}=&\sum_{r=1}^{k}~\sum_{\substack{\alpha_0+\cdots+\alpha_r=q-r\\\alpha_0,\ldots,\alpha_r\geq 0}}~\sum_{\substack{l_1+\cdots+l_{r}=k\\l_1,\ldots,l_{r}\geq 1}}\dfrac{k!}{l_1!\cdots l_r!}~V^{\alpha_0}\PF \left((iM)^{l_1}V\right)V^{\alpha_1}\\
	&\hspace*{3in}\times\cdots \times\PF \left((iM)^{l_{r}}V\right)V^{\alpha_r},
\end{align}
and 
\begin{align}\label{PCeq10}
	\nonumber\dfrac{d^k}{ds^k}\Big|_{s=0}\{\phi_q(T_s)\}=&\sum_{r=1}^{k}~\sum_{\substack{\alpha_0+\cdots+\alpha_r=q-r\\\alpha_0,\ldots,\alpha_r\geq 0}}~\sum_{\substack{l_1+\cdots+l_{r}=k\\l_1,\ldots,l_{r}\geq 1}}\dfrac{k!}{l_1!\cdots l_r!}~{T_0}^{\alpha_0}\PH \left((iM)^{l_1}W\right){T_0}^{\alpha_1}\\
	&\hspace*{3in}\times\cdots \times\PH \left((iM)^{l_{r}}W\right){T_0}^{\alpha_r}.
\end{align}
Now we denote $X_r:=V^{\alpha_0}\PF \left((iM)^{l_1}V\right)V^{\alpha_1}\cdots \PF \left((iM)^{l_{r}}V\right)V^{\alpha_r}$, where $\alpha_j\geq 0$ for $0\leq j\leq r$, and $l_{j'}\geq 1$ for $1\leq j'\leq r$, and $r\geq 1$. Next we require the block matrix representations of $M^n$ and $V^n$ on the space $\mathscr{F}\oplus\DT:=\HDTzeros\oplus\mathscr{H}\oplus \mathcal{D}_{T_0}\oplus S_{\DTzero}\HDTzero\oplus\mathcal{D}_{T_1}$ for any $n\in \mathbb{N}$ and they are the following:
\begin{align}\label{PCeq12}
	M^n=	\begin{bmatrix}
		0&0&0&0&0\\
		0&\bm{\ast}&\bm{\ast}&0&\bm{\ast}\\
		0&\bm{\ast}&\bm{\ast}&0&\bm{\ast}\\
		0&0&0&0&0\\
		0&\bm{\ast}&\bm{\ast}&0&\bm{\ast}
	\end{bmatrix}:
	\begin{bmatrix}
		\HDTzeros\\
		\mathscr{H}\\
		\mathcal{D}_{T_0}\\
		S_{\DTzero}\HDTzero\\
		\mathcal{D}_{T_1}	
	\end{bmatrix}\to \begin{bmatrix}
		\HDTzeros\\
		\mathscr{H}\\
		\mathcal{D}_{T_0}\\
		S_{\DTzero}\HDTzero\\
		\mathcal{D}_{T_1}	
	\end{bmatrix},
\end{align}
and 
\begin{align}\label{PCeq13}
\nonumber	V^n&=\begin{bmatrix}
		{S_{\DTzeros}^{*n}}&0&0\\
		\bm{\ast}&T_0^n&0\\
		\bm{\ast}&L_n&S_{\DTzero}^n
	\end{bmatrix}:
	\begin{bmatrix}
		\HDTzeros\\
		\mathscr{H}\\
		\HDTzero	
	\end{bmatrix}\to \begin{bmatrix}
		\HDTzeros\\
		\mathscr{H}\\
		\HDTzero		
	\end{bmatrix}\\
	&=\begin{bmatrix}
		{S_{\DTzeros}^{*n}}&0&0&0&0\\
		\bm{\ast}&T_0^n&0&0&0\\
		\bm{\ast}&D_{T_0}T_0^{n-1}&0&0&0\\
		\bm{\ast}&\bm{\ast}&S_{\DTzero}^n&S_{\DTzero}^n&0\\
		0&0&0&0&0\\
	\end{bmatrix}:
	\begin{bmatrix}
		\HDTzeros\\
		\mathscr{H}\\
		\mathcal{D}_{T_0}\\
		S_{\DTzero}\HDTzero\\
		\mathcal{D}_{T_1}	
	\end{bmatrix}\to \begin{bmatrix}
		\HDTzeros\\
		\mathscr{H}\\
		\mathcal{D}_{T_0}\\
		S_{\DTzero}\HDTzero\\
		\mathcal{D}_{T_1}	
	\end{bmatrix},
\end{align}
where $\ast$ stands for some non-zero entries and $L_n=D_{T_0}T_0^{n-1}+S_{\DTzero}L_{n-1}, L_0=0, n\geq 1$. Therefore using the structures \eqref{PCeq12} and \eqref{PCeq13} of $M^n$ and $V^n$ respectively we conclude that
\begin{align}\label{PCeq14}
\nonumber	\PH X_1\Big|_{\mathscr{H}}=&\PH V^{\alpha_0}\PF ((iM)^{l_1}V)V^{\alpha_1}\Big|_{\mathscr{H}}=\PH V^{\alpha_0}\PF ((iM)^{l_1} P_{S_{\DTzero}\HDTzero}^\perp V)V^{\alpha_1}\Big|_{\mathscr{H}}\\
	\nonumber=&\PH V^{\alpha_0}\PF ((iM)^{l_1} P_{\mathscr{H}\oplus\mathcal{D}_{T_0}} V)V^{\alpha_1}\Big|_{\mathscr{H}}=\PH V^{\alpha_0}P_{\mathscr{H}\oplus\mathcal{D}_{T_0}} ((iM)^{l_1} P_{\mathscr{H}\oplus\mathcal{D}_{T_0}} V)V^{\alpha_1}\Big|_{\mathscr{H}}\\
	=& T_0^{\alpha_0}\PH((iM)^{l_1} P_{\mathscr{H}\oplus\mathcal{D}_{T_0}} V)V^{\alpha_1}\Big|_{\mathscr{H}}=T_0^{\alpha_0}\PH ((iM)^{l_1}W)T_0^{\alpha_1}.
\end{align}
Again for $X_2$, we have
\begin{align}\label{PCeq15}
\nonumber	\PH X_2\Big|_{\mathscr{H}}=&\PH V^{\alpha_0}\PF ((iM)^{l_1}V)V^{\alpha_1}\PF((iM)^{l_2}V)V^{\alpha_2}\Big|_{\mathscr{H}}\\
\nonumber	=&\PH V^{\alpha_0}P_{\mathscr{H}\oplus\mathcal{D}_{T_0}} ({(iM)^{l_1}}P_{S_{\DTzero}\HDTzero}^\perp V)V^{\alpha_1}P_{\mathscr{H}\oplus\mathcal{D}_{T_0}}((iM)^{l_2}P_{S_{\DTzero}\HDTzero}^\perp V)V^{\alpha_2}\Big|_{\mathscr{H}}\\
\nonumber	=&\PH V^{\alpha_0}\PH ((iM)^{l_1}P_{S_{\DTzero}\HDTzero}^\perp V)V^{\alpha_1}\PH((iM)^{l_2}P_{S_{\DTzero}\HDTzero}^\perp V)V^{\alpha_2}\Big|_{\mathscr{H}}\\
	=&T_0^{\alpha_0}\PH ((iM)^{l_1}W)T_0^{\alpha_1}\PH ((iM)^{l_2}W)T_0^{\alpha_2},
\end{align}
and similarly for $r\geq 3$ we have
\begin{align}\label{PCeq16}
	\PH X_r\Big|_{\mathscr{H}}={T_0}^{\alpha_0}\PH \left((iB)^{l_1}W\right){T_0}^{\alpha_1}\cdots \PH \left((iB)^{l_{r}}W\right){T_0}^{\alpha_r}.
\end{align}
Therefore combining equations \eqref{PCeq9}, \eqref{PCeq10}, \eqref{PCeq14}, \eqref{PCeq15} and \eqref{PCeq16} we get
\begin{align}\label{PCeq17}
	\phi_q(T_1)-\phi_q(T_0)-\sum_{k=1}^{n-1}\dfrac{1}{k!}\dfrac{d^k}{ds^k}\Big|_{s=0}\phi_q(T_s) =\PH\left(\phi_q(T)-\phi_q(V)-\sum_{k=1}^{n-1}\dfrac{1}{k!}\dfrac{d^k}{ds^k}\Big|_{s=0}\phi_q(V_s)\right)\Bigg|_{\mathscr{H}},
\end{align}
for all $q\in \mathbb{N}$. Again by using the facts
$\dfrac{d^k}{ds^k}\Big|_{s=0}\{\phi_q(T_s)\} =\left(\dfrac{d^k}{ds^k}\Big|_{s=0}\{\phi_{-q}(T_s)\}\right)^*$ and 
$\dfrac{d^k}{ds^k}\Big|_{s=0}\{\phi_q(V_s)\} =\left(\dfrac{d^k}{ds^k}\Big|_{s=0}\{\phi_{-q}(V_s)\}\right)^*$, and applying Lemma \ref{lm1} together with the same analysis as above we also get for $q\in \mathbb{Z},~q<0$ that
\begin{align}\label{PCeq18}
	\phi_q(T_1)-\phi_q(T_0)-\sum_{k=1}^{n-1}\dfrac{1}{k!}\dfrac{d^k}{ds^k}\Big|_{s=0}\phi_q(T_s) =\PH\left(\phi_q(T)-\phi_q(V)-\sum_{k=1}^{n-1}\dfrac{1}{k!}\dfrac{d^k}{ds^k}\Big|_{s=0}\phi_q(V_s)\right)\Bigg|_{\mathscr{H}}.
\end{align}
Therefore using \eqref{PCeq4}, \eqref{PCeq17} and \eqref{PCeq18} we conclude 
$\left\{\phi_q(T_1)-\phi_q(T_0)-\sum\limits_{k=1}^{n-1}\dfrac{1}{k!}\dfrac{d^k}{ds^k}\Big|_{s=0}\phi_q(T_s)\right\}$
is a trace class operator and 
\begin{equation} \label{PCeq19}
	\begin{split}
		&\textup{Tr}\left\{\phi_q(T_1)-\phi_q(T_0)-\sum_{k=1}^{n-1}\dfrac{1}{k!}\dfrac{d^k}{ds^k}\Big|_{s=0}\phi_q(T_s)\right\}\\ &\hspace{1in}=\textup{Tr}\left\{\PH\left(\phi_q(T)-\phi_q(V)-\sum_{k=1}^{n-1}\dfrac{1}{k!}\dfrac{d^k}{ds^k}\Big|_{s=0}\phi_q(V_s)\right)\Bigg|_{\mathscr{H}}\right\},\forall q\in \mathbb{Z}.
	\end{split}
\end{equation}
Furthermore, analyzing the structures of $M^n$ and $V^n$ as in \eqref{PCeq12} and \eqref{PCeq13} respectively we conclude
\begin{align*}
	P_{\mathscr{F}\ominus\mathscr{H}}X_1\Big|_{\mathscr{F}\ominus\mathscr{H}}=&P_{\mathscr{F}\ominus\mathscr{H}}V^{\alpha_0}\PF ((iM)^{l_1}V)V^{\alpha_1}\Big|_{\mathscr{F}\ominus\mathscr{H}}\\
	=&P_{\mathscr{F}\ominus\mathscr{H}}V^{\alpha_0}P_{\mathscr{H}\oplus\mathcal{D}_{T_0}} ((iM)^{l_1} P_{\mathscr{H}\oplus\mathcal{D}_{T_0}}V) V^{\alpha_1}\Big|_{\mathscr{F}\ominus\mathscr{H}}\\
	=&P_{\mathscr{F}\ominus\mathscr{H}}V^{\alpha_0}P_{\mathscr{H}\oplus\mathcal{D}_{T_0}} ((iM)^{l_1} P_{\mathscr{H}\oplus\mathcal{D}_{T_0}}V) V^{\alpha_1}P_{\HDTzeros}\Big|_{\mathscr{F}\ominus\mathscr{H}}\\
	=&P_{\HDTzero}V^{\alpha_0}P_{\mathscr{H}\oplus\mathcal{D}_{T_0}} ((iM)^{l_1} P_{\mathscr{H}\oplus\mathcal{D}_{T_0}}V) V^{\alpha_1}P_{\HDTzeros}\Big|_{\mathscr{F}\ominus\mathscr{H}},
\end{align*}
and for $r\geq 2$,
\begin{align*}
	 P_{\mathscr{F}\ominus\mathscr{H}}X_r\Big|_{\mathscr{F}\ominus\mathscr{H}}=&P_{\mathscr{F}\ominus\mathscr{H}}V^{\alpha_0}\PF ((iM)^{l_1}V)V^{\alpha_1}\cdots\PF((iM)^{l_{r-1}}V)V^{\alpha_{r-1}}\PF((iM)^{l_{r}}V)V^{\alpha_{r}}\Big|_{\mathscr{F}\ominus\mathscr{H}}\\
	 =&P_{\mathscr{F}\ominus\mathscr{H}}V^{\alpha_0}P_{\mathscr{H}\oplus\mathcal{D}_{T_0}} ((iM)^{l_1} P_{\mathscr{H}\oplus\mathcal{D}_{T_0}}V) V^{\alpha_1}P_{\mathscr{H}\oplus\mathcal{D}_{T_0}}\cdot\cdot P_{\mathscr{H}\oplus\mathcal{D}_{T_0}}((iM)^{l_{r-1}} P_{\mathscr{H}\oplus\mathcal{D}_{T_0}}V)\\
	&\hspace*{.5in}\times V^{\alpha_{r-1}}P_{\mathscr{H}\oplus\mathcal{D}_{T_0}}((iM)^{l_{r}} P_{\mathscr{H}\oplus\mathcal{D}_{T_0}}V) V^{\alpha_{r}}\Big|_{\mathscr{F}\ominus\mathscr{H}}\\
	 =&P_{\HDTzero}V^{\alpha_0}P_{\mathscr{H}\oplus\mathcal{D}_{T_0}} ((iM)^{l_1} P_{\mathscr{H}\oplus\mathcal{D}_{T_0}}V) V^{\alpha_1}P_{\mathscr{H}\oplus\mathcal{D}_{T_0}}\cdot \cdot P_{\mathscr{H}\oplus\mathcal{D}_{T_0}}((iM)^{l_{r-1}} P_{\mathscr{H}\oplus\mathcal{D}_{T_0}}V)\\
	&\hspace{.5in}\times V^{\alpha_{r-1}}P_{\mathscr{H}\oplus\mathcal{D}_{T_0}}((iM)^{l_{r}} P_{\mathscr{H}\oplus\mathcal{D}_{T_0}}V) V^{\alpha_{r}}P_{\HDTzeros}\Big|_{\mathscr{F}\ominus\mathscr{H}},
\end{align*}
which implies that the operator $P_{\mathscr{F}\ominus\mathscr{H}}\Bigg(\phi_q(T)-\phi_q(V)-\sum\limits_{k=1}^{n-1}\dfrac{1}{k!}\dfrac{d^k}{ds^k}\Big|_{s=0}\phi_q(V_s)\Bigg)\Bigg|_{\mathscr{F}\ominus\mathscr{H}}$ maps $\HDTzeros\oplus \textbf{0}\oplus \textbf{0}$ to $\textbf{0}\oplus \textbf{0}\oplus\HDTzero$ for $q\in \mathbb{N}$ and the operator  $P_{\mathscr{F}\ominus\mathscr{H}}\Bigg(\phi_q(T)-\phi_q(V)$ $-\sum\limits_{k=1}^{n-1}\dfrac{1}{k!}\dfrac{d^k}{ds^k}\Big|_{s=0}\phi_q(V_s)\Bigg)\Bigg|_{\mathscr{F}\ominus\mathscr{H}}$ maps $ \textbf{0}\oplus \textbf{0}\oplus\HDTzero$ to $\HDTzeros\oplus \textbf{0}\oplus \textbf{0}$ for $q\in \mathbb{Z},~q<0$. The above observations immediately yields that 
\begin{align}\label{PCeq20}
	\textup{Tr}\left\{P_{\mathscr{F}\ominus\mathscr{H}}\left(\phi_q(T)-\phi_q(V)-\sum_{k=1}^{n-1}\dfrac{1}{k!}\dfrac{d^k}{ds^k}\Big|_{s=0}\phi_q(V_s)\right)\Bigg|_{\mathscr{F}\ominus\mathscr{H}} \right\}=0, ~~\forall q\in \mathbb{Z}.
\end{align}
Therefore combining equations \eqref{PCeq19} and \eqref{PCeq20} we get
\begin{equation*}
	\textup{Tr}\left\{\phi_q(T_1)-\phi_q(T_0)-\sum_{k=1}^{n-1}\dfrac{1}{k!}\dfrac{d^k}{ds^k}\Big|_{s=0}\phi_q(T_s)\right\} =\textup{Tr}\left\{\phi_q(T)-\phi_q(V)-\sum_{k=1}^{n-1}\dfrac{1}{k!}\dfrac{d^k}{ds^k}\Big|_{s=0}\phi_q(V_s)\right\}
\end{equation*}
for all $ q\in \mathbb{Z}$ and hence \eqref{PCeq7} follows. Thus the conclusion of the theorem follows by combining equations \eqref{PCeq5} and \eqref{PCeq7}. This completes the proof. 

\end{proof}

\section{Higher order Trace formula for pair of maximal dissipative operators}\label{maxdis}
In this section our main aim is to prove the trace formula for pairs of maximal dissipative operators as an application of our main theorem in the previous section.   
We start with the section by recalling the definition  of dissipative operator. Let $A:\hil\to\hil$ be a linear operator (need not be bounded) with dense domain $\textup{Dom}(A)$ called dissipative if $\textup{Im}\langle Ah,h\rangle\leq 0$ for all $h\in\textup{Dom}(A)$. A dissipative operator is called maximal if it has no proper dissipative extension. It is well known that the Cayley transform of a maximal dissipative operator $A$ is a contraction $T:\hil\to\hil$ given by $T=-(A+i)(A-i)^{-1}$ such that $\ker T=\ker (A+i)$ and $\ker T^*=\ker (A^*-i)$. Furthermore, the following pieces of information are also enlisted in \cite{Mor}, but for reader convenience and the self-containment of the article, we are providing it here as well.
\begin{align}
\label{MDeq1}	&D_{T}=2|(-\textup{Im} A)^{1/2}(A-i)^{-1}|,\quad D_{T^*}=2|(-\textup{Im} A)^{1/2}(A^*+i)^{-1}|,\\
\label{MDeq2}	&\mathcal{D}_{T}=\overline{\left( (A^*+i)^{-1}(\textup{Im} A)\hil  \right)},\quad \mathcal{D}_{T^*}=\overline{\left( (A-i)^{-1}(\textup{Im} A)\hil \right)}.
\end{align}
In the case of dissipative operator we need a different class of functions different from the class considered in the last two sections. Let us consider the following class:
\[\mathcal{R}_n:=\left\{ \psi:\mathbb{R}\to\mathbb{C} \text{ such that } \psi(\lambda)=\phi\left(\dfrac{i+\lambda}{i-\lambda}\right) \text{ for some } \phi\in\mathcal{F}_n(\mathbb{T})  \right\}.\] 
Next we define $\psi_{\pm}$ using $\phi_{\pm}$ in a similar way as we have done in \eqref{INTeq1} and hence we obtain the decomposition $\psi(\lambda)=\psi_+(\lambda)+\psi_-(-\lambda)$. In other words, if $\psi\in\mathcal{R}_n$, then $\psi(\lambda)=\phi\left(\dfrac{i+\lambda}{i-\lambda}\right)$ for some $\phi\in\mathcal{F}_n(\mathbb{T})$, and 
\begin{align*}
	\psi_+(\lambda)=\phi_+\left(\dfrac{i+\lambda}{i-\lambda}\right) \text{ and }  \psi_-(\lambda)=\phi_-\left(\dfrac{i+\lambda}{i-\lambda}\right).
\end{align*}
Now we set 
\begin{align*}
	\psi_+(A)=\phi_+(T), \psi_+(-A^*)=\phi_-(T^*) \text{, and } \psi(A)=\psi_+(A)+\psi_-(-A^*). 
\end{align*}
The following lemma is essential to prove the main theorem in this section.
\begin{lma}\label{lm2}
	Let $\psi\in\mathcal{R}_n$ be such that $\psi(\lambda)=\phi\left(\dfrac{i+\lambda}{i-\lambda}\right)$ for some $\phi\in\mathcal{F}_n(\mathbb{T})$. Now if we substitute $e^{it}=\dfrac{i+\lambda}{i-\lambda}$, then $\phi(e^{it})=\psi(\lambda),~\lambda=-\tan\dfrac{t}{2}$, and for all $1\leq q\leq n-1$,
	\begin{align}\label{MDeq3}
		\dfrac{d^q}{dt^q}\{\phi(e^{it})\}=\left(\sum_{k=0}^{q-1}p_{k,q}(\lambda)~\psi^{(q-k)}(\lambda)\right)\dfrac{d\lambda}{dt},
	\end{align}
 where $p_{k,q}$ are polynomials in $\lambda$ of degree $(2(q-1)-k)$ and it is given recursively as follows
	\begin{align*}
		p_{k,q}(\lambda)=\begin{cases}
			(-1/2)(1+\lambda^2)~p_{0,q-1}(\lambda) &\text{ for } k=0,\\
			(-1/2)\left\{(1+\lambda^2)\left(p_{k,q-1}(\lambda)+p^{(1)}_{k-1,q-1}(\lambda)\right)+2\lambda p_{k-1,q-1}(\lambda) \right\} &\text{ for } 1\leq k\leq q-2,\\
			(-1/2)\left[(1+\lambda^2)p^{(1)}_{q-2,q-1}(\lambda)+2\lambda p_{q-2,q-1}(\lambda) \right]&\text{ for } k=q-1,
		\end{cases}
	\end{align*} 
and $p_{0,1}(\lambda)=1$.
\end{lma}
\begin{proof}
	We prove the identity \eqref{MDeq3} by principle of mathematical induction. For $q=1$, $\dfrac{d}{dt}\{\phi(e^{it})\}=\left(\dfrac{d}{d\lambda}\{\psi(\lambda)\}\right)\dfrac{d\lambda}{dt}$ and hence \eqref{MDeq3} is true for $q=1$. Suppose \eqref{MDeq3} is true for $q=m<n-1$, that is
	\begin{align*}
		\dfrac{d^m}{dt^m}\{\phi(e^{it})\}=&\left(\sum_{k=0}^{m-1}p_{k,m}(\lambda)\psi^{(m-k)}(\lambda)\right)\dfrac{d\lambda}{dt}.
	\end{align*}
Now we will show that \eqref{MDeq3} is also true for $q=m+1$. Note that
\begin{align*}
		\dfrac{d^{m+1}}{dt^{m+1}}\{\phi(e^{it})\}=&\sum_{k=0}^{m-1}\Big[\Big(p^{(1)}_{k,m}(\lambda)\psi^{(m-k)}(\lambda)+p_{k,m}(\lambda)\psi^{(m+1-k)}(\lambda)\Big)\dfrac{d\lambda}{dt}+(-\lambda)p_{k,m}(\lambda)\psi^{(m-k)}(\lambda)\Big]\dfrac{d\lambda}{dt}\\
		=&~p_{0,m}(\lambda)\psi^{(m+1)}(\lambda)\left(\dfrac{d\lambda}{dt}\right)^2+\sum_{k=0}^{m-2}\Bigg[  p_{k+1,m}(\lambda)\dfrac{d\lambda}{dt}+p^{(1)}_{k,m}(\lambda)\dfrac{d\lambda}{dt} +(-\lambda)p_{k,m}(\lambda)\Bigg]\\
		&\times\psi^{(m-k)}(\lambda)\dfrac{d\lambda}{dt}+\left[p^{(1)}_{m-1,m}(\lambda)\dfrac{d\lambda}{dt}+(-\lambda) p_{m-1,m}(\lambda)\right]\psi^{(1)}(\lambda)\dfrac{d\lambda}{dt}\\
		=&~p_{0,m}(\lambda)\psi^{(m+1)}(\lambda)\left(\dfrac{d\lambda}{dt}\right)^2+\sum_{k=1}^{m-1}\Bigg[  p_{k,m}(\lambda)\dfrac{d\lambda}{dt}+p^{(1)}_{k-1,m}(\lambda)\dfrac{d\lambda}{dt} +(-\lambda)p_{k-1,m}(\lambda)\Bigg]\\
		&\times\psi^{(m+1-k)}(\lambda)\dfrac{d\lambda}{dt}+\left[p^{(1)}_{m-1,m}(\lambda)\dfrac{d\lambda}{dt}+(-\lambda) p_{m-1,m}(\lambda)\right]\psi^{(1)}(\lambda)\dfrac{d\lambda}{dt}\\
		=&\left(\sum_{k=0}^{m}p_{k,m+1}(\lambda)\psi^{(m+1-k)}(\lambda)\right)\dfrac{d\lambda}{dt},
	\end{align*} 
where
	\begin{align*}
		p_{k,m+1}(\lambda)=\begin{cases}
			(-1/2)(1+\lambda^2)~p_{0,m}(\lambda) &\text{ for } k=0,\\
			(-1/2)\left\{(1+\lambda^2)\left(p_{k,m}(\lambda)+p^{(1)}_{k-1,m}(\lambda)\right)+2\lambda p_{k-1,m}(\lambda) \right\} &\text{ for } 1\leq k\leq m-1,\\
			(-1/2)\left[(1+\lambda^2)p^{(1)}_{m-1,m}(\lambda)+2\lambda p_{m-1,m}(\lambda) \right]&\text{ for } k=m,
		\end{cases}
	\end{align*} and degree of $p_{k,m+1}$ is $(2((m+1)-1)-k)$, and hence \eqref{MDeq3} is true for $q=m+1$. Therefore the result follows by principle of mathematical induction. This completes the proof.
\end{proof}
Now we are in a position to state and prove our main result in this section. It is important to note that we make the hypothesis of our next theorem in such a way so that we can apply Theorem~\ref{thconcon} to achieve our goal.

\begin{thm}\label{thdisdis}
	Let $A_0$ and $A_1$ be two maximal dissipative operators on $\hil$ such that
	\vspace{0.1in}
	
		$(i)$  $\dim \ker (A_j+i)=\dim\ker (A^*_j-i),$ for $j=0,1$,
		\vspace{0.1in}
		
		$(ii)$ $(A_1-i)^{-1}-(A_0-i)^{-1}\in\bnh$, and
		\vspace{0.1in}
		 
		$(iii)$ \textup{Im} $A_j=\dfrac{A_j-A_j^*}{2i}\in\mathcal{B}_{n/2}(\hil)$ for $j=0,1$.
		\vspace{0.1in}
		
		Let $T_0=-(A_0+i)(A_0-i)^{-1}$ and $T_1=-(A_1+i)(A_1-i)^{-1}$ be the corresponding contractions obtained by the Cayley transform of  maximal dissipative operators $A_0$ and $A_1$ respectively. Set $A_s=\Big(i-2i(T_s+1)^{-1}\Big)$, where $T_s$ as in \eqref{PCeqpath}.
		Then for $\psi\in \mathcal{R}_n$, 
		\begin{align*}
		\left\{\psi(A_1)-\psi(A_0)-\sum_{k=0}^{n-1}\dfrac{1}{k!}\dfrac{d^k}{ds^k}\Big|_{s=0}\psi(A_s)\right\}\in\boh,
		\end{align*} 
		and there exists an $L^1(\mathbb{T})$-function  $\xi_n$ 
		(unique up to additive constant) depend only on $n,A_1$ and $A_0$ such that 
		\begin{align*}
			\textup{Tr}\left\{\psi(A_1)-\psi(A_0)-\sum_{k=1}^{n-1}\dfrac{1}{k!}\dfrac{d^k}{ds^k}\Big|_{s=0}\psi(A_s) \right\}=\int_{0}^{2\pi}\dfrac{d^n}{dt^n}\{\phi(e^{it})\}\xi_n(t)dt,
		\end{align*}
where $\psi(\lambda)=\phi\left(\dfrac{i+\lambda}{i-\lambda}\right)$ for some $\phi\in\mathcal{F}_n(\mathbb{T})$ and $\lambda\in \mathbb{R}$. Moreover, if  $\psi\in \mathcal{S}(\mathbb{R})$ (Schwartz class of functions on $\mathbb{R}$), then there exists $\eta_n\in L^1\Big(\mathbb{R},(1+\lambda^2)^{-n}d\lambda \Big)$ such that 
\begin{align*}
	\textup{Tr}\left\{\psi(A_1)-\psi(A_0)-\sum_{k=0}^{n-1}\dfrac{1}{k!}\dfrac{d^k}{ds^k}\Big|_{s=0}\psi(A_s) \right\}=\int_{-\infty}^{\infty}\psi^{(n)}(\lambda)\eta_n(\lambda)d\lambda.
\end{align*}		
\end{thm}
\begin{proof}
Let $T_j=-(A_j+i)(A_j-i)^{-1}$ be the contraction obtained via the Cayley transform of a maximal dissipative operator $A_j$ and hence  $\ker T_j=\ker (A_j+i)$ and $\ker T_j^*=\ker (A_j^*-i)$ for $j=0,1$. Furthermore, note that 
\[
T_1-T_0=-2i\left[(A_1-i)^{-1}-(A_0-i)^{-1}\right].
\]
Therefore using the hypothesis $(i)$, $(ii)$ and $(iii)$ we conclude that the pair of contractions $(T_0,T_1)$ on $\mathscr{H}$ satisfies the hypothesis $(i)$ and $(ii)$ of Theorem~\ref{thconcon}. Let $V_j$ be the unitary operator on $\mathscr{H}$ such that $(A_j+i)(A_j-i)^{-1}=V_j|(A_j+i)(A_j-i)^{-1}|$ for $j=0,1$. 
Thus by applying Theorem~\ref{thconcon} corresponding to the pair $(T_0,T_1)$ we get for $\phi\in \mathcal{F}_n(\mathbb{T})$,
\begin{align}\label{MDeq4}
	\left\{\phi(T_1)-\phi(T_0)-\sum_{k=1}^{n-1}\dfrac{1}{k!}\dfrac{d^k}{ds^k}\Big|_{s=0}\phi(T_s) \right\}\in\boh,
\end{align} 
and there exists an $L^1(\mathbb{T})$-function  $\xi_n$ 
(unique up to additive constant) depend only on $n,T_1$ and $T_0$ such that 
\begin{align}\label{MDeq5}
	\textup{Tr}\left\{\phi(T_1)-\phi(T_0)-\sum_{k=1}^{n-1}\dfrac{1}{k!}\dfrac{d^k}{ds^k}\Big|_{s=0}\phi(T_s) \right\}=\int_{0}^{2\pi}\dfrac{d^n}{dt^n}\{\phi(e^{it})\}\xi_n(t)dt,
\end{align}
where 
\begin{align}
T_s=\PH e^{isM}\begin{bmatrix}
	0\\
	-(A_0+i)(A_0-i)^{-1}\\
	|2(-\textup{Im} A_0)^{1/2}(A_0-i)^{-1}|\\
	0
\end{bmatrix},~~s\in [0,1],
\end{align}
and $M$ is a self-adjoint operator on $\mathscr{F}\oplus \overline{\left( (A_1^*+i)^{-1}(\textup{Im} A_1)\hil  \right)}$ such that  $$\mathscr{F}= \mathbf{H}^2_{\overline{\left( (A_0-i)^{-1}(\textup{Im} A_0)\hil\right)}}(\D)\oplus\hil\oplus\mathbf{H}^2_{ \overline{\left( (A_0^*+i)^{-1}(\textup{Im} A_0)\hil  \right)}}(\mathbb{D}),$$  $\sigma(M)\subseteq (-\pi,\pi]$,
$M\in \mathcal{B}_n(\mathscr{F}\oplus \overline{\left( (A_1^*+i)^{-1}(\textup{Im} A_1)\hil  \right)})$ and $\mathcal{M}=e^{iM}$. Furthermore, the block matrix representation of $\mathcal{M}$ on $\mathscr{F}\oplus \overline{\left( (A_1^*+i)^{-1}(\textup{Im} A_1)\hil  \right)}$ is the following:
\begin{align*}
	\mathcal{M}:=\begin{bmatrix}
		I&0&0&0\\
		0&\mathcal{M}_{22}&\mathcal{M}_{23}&\mathcal{M}_{24}\\
		0&\mathcal{M}_{32}& \mathcal{M}_{33}&0\\
		0&\mathcal{M}_{42}&\mathcal{M}_{43}&\mathcal{M}_{44}
	\end{bmatrix}:
	\begin{bmatrix}
	\mathbf{H}^2_{\overline{\left( (A_0-i)^{-1}(\textup{Im} A_0)\hil\right)}}(\D)\\
		\mathscr{H}\\
		\mathbf{H}^2_{ \overline{\left( (A_0^*+i)^{-1}(\textup{Im} A_0)\hil  \right)}}(\mathbb{D})\\
		\overline{\left( (A_1^*+i)^{-1}(\textup{Im} A_1)\hil  \right)}	
	\end{bmatrix}\to \begin{bmatrix}
	\mathbf{H}^2_{\overline{\left( (A_0-i)^{-1}(\textup{Im} A_0)\hil\right)}}(\D)\\
	\mathscr{H}\\
	\mathbf{H}^2_{ \overline{\left( (A_0^*+i)^{-1}(\textup{Im} A_0)\hil  \right)}}(\mathbb{D})\\
	\overline{\left( (A_1^*+i)^{-1}(\textup{Im} A_1)\hil  \right)}	
\end{bmatrix},
\end{align*}
where $\mathcal{M}_{22}=(A_1+i)(A_1-i)^{-1}(A_0^*+i)^{-1}(A_0^*-i)$, $\mathcal{M}_{23}=-(A_1+i)(A_1-i)^{-1}2|(-\textup{Im} A_0)^{1/2}$ $(A_0-i)^{-1}|P_{\overline{\left( (A_0^*+i)^{-1}(\textup{Im} A_0)\hil  \right)}}$, $\mathcal{M}_{24}=2|(-\textup{Im} A_1)^{1/2}(A_1^*+i)^{-1}|V_1$, $\mathcal{M}_{32}=2V_0^*|(-\textup{Im} A_0^*)^{1/2}(A_0^*+i)^{-1}|$, $\mathcal{M}_{33}=|(A_0+i)(A_0-i)^{-1}|P_{\overline{\left( (A_0^*+i)^{-1}(\textup{Im} A_0)\hil  \right)}} + (I-P_{\overline{\left( (A_0^*+i)^{-1}(\textup{Im} A_0)\hil  \right)}})$, $\mathcal{M}_{42}=-2|(-\textup{Im} A_1)^{1/2}$ $(A_1-i)^{-1}|(A_0^*+i)^{-1}(A_0^*-i)$, $\mathcal{M}_{43}=4|(-\textup{Im} A_1)^{1/2}(A_1-i)^{-1}||(-\textup{Im} A_0)^{1/2}(A_0-i)^{-1}|$ \\$P_{\overline{\left( (A_0^*+i)^{-1}(\textup{Im} A_0)\hil  \right)}}$, and $\mathcal{M}_{44}=(A_1^*+i)^{-1}(A_1^*-i)V_1$. Now it easy to observe that for  $\psi(\lambda)=\phi\left(\dfrac{i+\lambda}{i-\lambda}\right)\in\mathcal{R}_n$, where $\phi\in \mathcal{F}_n(\mathbb{T})$,
\begin{align}\label{MDeq6}
	\psi(A_1)-\psi(A_0)-\sum_{k=0}^{n-1}\dfrac{1}{k!}\dfrac{d^k}{ds^k}\Big|_{s=0}\psi(A_s)=\phi(T_1)-\phi(T_0)-\sum_{k=1}^{n-1}\dfrac{1}{k!}\dfrac{d^k}{ds^k}\Big|_{s=0}\phi(T_s),
\end{align}
where $A_s=\left(i-2i(T_s+1)^{-1}\right)$. Therefore using equations \eqref{MDeq4}, \eqref{MDeq5} and \eqref{MDeq6} we conclude that 
\begin{align*}
	\left\{\psi(A_1)-\psi(A_0)-\sum_{k=0}^{n-1}\dfrac{1}{k!}\dfrac{d^k}{ds^k}\Big|_{s=0}\psi(A_s)\right\}\in\boh,
\end{align*} 
and there exists an $L^1(\mathbb{T})$-function  $\xi_n$ 
(unique up to additive constant) depend only on $n,A_1$ and $A_0$ such that 
\begin{align*}
	\textup{Tr}\left\{\psi(A_1)-\psi(A_0)-\sum_{k=1}^{n-1}\dfrac{1}{k!}\dfrac{d^k}{ds^k}\Big|_{s=0}\psi(A_s) \right\}=\int_{0}^{2\pi}\dfrac{d^n}{dt^n}\{\phi(e^{it})\}\xi_n(t)dt,
\end{align*}
which by applying Lemma~\ref{lm2} yields that
\begin{align}\label{MDeq7}
	 \textup{Tr}\left\{\psi(A_1)-\psi(A_0)-\sum_{k=0}^{n-1}\dfrac{1}{k!}\dfrac{d^k}{ds^k}\Big|_{s=0}\psi(A_s) \right\}=\int_{-\infty}^{\infty}\left(\sum_{k=0}^{n-1}p_{k,n}(\lambda)\psi^{(n-k)}(\lambda)\right)\eta_{_n}(\lambda)d\lambda,	
\end{align} where $\eta_n(\lambda)=\xi_n(t)=\xi_n(-2tan^{-1}(\lambda))$ and $\eta_n\in L^1\left(\mathbb{R},(1+\lambda^2)^{-1}d\lambda \right)$. In particular, if we consider $\psi\in \mathcal{S}(\mathbb{R})\subset\mathcal{R}_n$, then by performing integration by-parts \eqref{MDeq7} becomes 
\begin{align*}
	\textup{Tr}\left\{\psi(A_1)-\psi(A_0)-\sum_{k=0}^{n-1}\dfrac{1}{k!}\dfrac{d^k}{ds^k}\Big|_{s=0}\psi(A_s) \right\}=\int_{-\infty}^{\infty}\psi^{(n)}(\lambda)\zeta_n(\lambda)d\lambda,
\end{align*}
where
\begin{align*}	
	\zeta_n(\lambda)=\left[ \sum_{k=0}^{n-1}(-1)^{k}\eta_{_k,_k}^n(\lambda)\right], \eta^n_{_0,_0}(\lambda)=p_{_0,_n}(\lambda),~\text{and}~
	\eta^n_{_j,_k}(\lambda)=&\begin{cases}
		\int_{0}^{\lambda}p_{_k,_n}(\lambda)\eta_{_n}(\lambda) &\text{ if } j=1,\\
		\int_{0}^{\lambda}\eta^n_{_{j-1},_k}(\lambda) &\text{ if } 2\leq j\leq k.
	\end{cases}
\end{align*}
This completes the proof. 
\end{proof}

		\section*{Acknowledgments}
	\textit{The research of the first named author is supported by the Mathematical Research Impact Centric
		Support (MATRICS) grant, File No :MTR/2019/000640, by the Science and Engineering Research Board (SERB), Department of Science $\&$ Technology (DST), Government of India. The second named author gratefully acknowledge the support provided by IIT Guwahati, Government of India.}

\end{document}